\documentclass[onefignum,onetabnum]{siamart190516}

\usepackage{braket,amsfonts}
\usepackage{amsmath,amssymb,amsxtra}
\usepackage{bm,mathrsfs,marvosym,eurosym}
\usepackage{mathtools,enumerate}
\usepackage{booktabs}
\usepackage{threeparttable}

\usepackage{bbm}
\usepackage{array}
\usepackage[normalem]{ulem}

\usepackage[caption=false]{subfig}
\newsiamremark{remark}{Remark}
\newsiamthm{condition}{condition}
\newtheorem{example}{Example}[section]

\numberwithin{figure}{section}
\numberwithin{table}{section}

\usepackage{algorithmic}

\usepackage{graphicx}
\graphicspath{{./pics/}}

 \renewcommand{\d}{\mathrm{d}}

\allowdisplaybreaks
\input{siam.sty}

\usepackage{color}
\newcommand{\red}[1]{{\color{red} #1}}

\renewcommand\sout[1]{\bgroup\markoverwith{\textcolor{red}{\rule[0.5ex]{2pt}{0.6pt}}}\ULon{#1}}

\def\d{{\mathrm d}}

\def\G{{\mathcal G}}
\def\F{{\mathcal F}}

\def\A{{\mathcal{A}}}
\def\H{{\mathcal{H}}}

\def \L{\mathcal{L}}
\def \t{\theta}
\def\bt{\boldsymbol{\t}}

\def\<{{\langle }}
\def\>{{\rangle }}

\definecolor{darkred}{RGB}{200,0,0}
\definecolor{darkblue}{RGB}{0,0,200}
\definecolor{darkgreen}{RGB}{0,160,0}

\title{Optimizing Coarse Propagators in Parareal Algorithms\thanks{The work of B. Jin is supported by UK EPSRC EP/V026259/1, Hong Kong RGC General Research Fund (14306423), and a start-up fund from The Chinese University of Hong Kong. The work of  Z. Zhou is supported by Hong Kong RGC General Research Fund (15303021) and an internal grant of Hong Kong Polytechnic University (Project ID: P0038888, Work Programme: ZVX3).}}
\author{Bangti Jin\thanks{Department of Mathematics, The Chinese University of Hong Kong, Shatin, New Territories, Hong Kong, P.R. China (\texttt{bangti.jin@gmail.com, b.jin@cuhk.edu.hk}).} \and Qingle Lin\thanks{School of Mathematics, Jilin University, Changchun 130012, P.R. China (\texttt{linql1020@mails.jlu.edu.cn})}
\and Zhi Zhou\thanks{Department of Applied Mathematics, The Hong Kong Polytechnic University, Kowloon, Hong Kong, P.R. China. (\texttt{zhizhou@polyu.edu.hk})}}
\date{\today}
		
\headers{Optimizing Coarse Propagators in Parareal Algorithms}{Bangti Jin, Qingle Lin and Zhi Zhou}

\begin{document}
	\maketitle

\begin{abstract}
The parareal algorithm represents an important class of parallel-in-time algorithms for solving evolution equations and has been widely applied in practice. To achieve effective speedup, the choice of the coarse propagator in the algorithm is vital. In this work, we investigate the use of {optimized} coarse propagators. Building upon the error estimation framework, we present a systematic procedure for constructing coarse propagators that enjoy desirable stability and consistent order. Additionally, we provide preliminary mathematical guarantees for the resulting parareal algorithm. Numerical experiments on a variety of settings, e.g., linear diffusion model, Allen-Cahn model, and viscous Burgers model, show that the optimizing procedure can significantly improve parallel efficiency when compared with the more ad hoc choice of some conventional and widely used coarse propagators.
\end{abstract}



\pagestyle{myheadings}
\thispagestyle{plain}

\section{Introduction}\label{sec:intro}
This paper is devoted to accelerating a class of parareal solvers for parabolic problems. Let $T >0$ be a fixed terminal time.  Consider the following initial value problem for $u \in C((0,T];D(\A))\cap C([0,T];\H)$:
\begin{equation}\label{eqn:pde}
		\left \{
		\begin{aligned}
			&u' (t) + \A u(t) = f(t), \quad 0<t<T,\\
			& u(0)=u^0 ,
		\end{aligned}\right .
\end{equation}
where $u^0\in \H$ is the initial data,  $f : [0,T] \to \H$ is a given forcing term, and $\A$ is a positive definite, self-adjoint, linear operator with a compact inverse, defined over a
Hilbert space $(\H, (\cdot , \cdot ))$ (with the induced norm $\|\cdot\|$) with its domain $D(\A)$ dense in $\H$. This model arises in a very broad range of practical applications, and its effective numerical simulation is of enormous importance and has received a lot of attention \cite{thomee2007galerkin}.

The numerical solution of the model \eqref{eqn:pde} commonly employs time discretization, and various time-stepping schemes have been proposed; see the monograph \cite{thomee2007galerkin} for an in-depth treatment. Typically these schemes proceed step by step, and the process is sequential in nature. This represents the major bottleneck in the numerical solution of the model \eqref{eqn:pde}, especially for long time. In the ground-breaking work \cite{LionsMadayTurinici:2001}, Lions, Maday and Turinici proposed the so-called parareal algorithm for solving evolution models in a parallel-in-time manner. The method has been established as a fundamental tool to utilize highly parallel systems when distributed computing in space saturates. So far this class of methods have been applied successfully to numerous diverse practical applications, e.g., option pricing \cite{BalMaday:2002,PagesPironneau:2016}, multiscale kinetics \cite{Engblom:2009,ArielKim:2016},
stochastic PDEs \cite{BrehierWang:2020} and fractional diffusion \cite{XuHesthaven:2015,LiHu:2021}.

The parareal algorithm is based on two solvers, one coarse propagator (CP), of low accuracy but fast enough to be applied sequentially, and one fine propagator (FP) with full accuracy, applied in parallel in time. At each iteration, the CP is run sequentially over each subdomain,
and then the FP is run in parallel as a correction. See Section \ref{ssec:parareal} for the details on how the CP and FP interact. The CP is usually based on one or a few steps of a standard numerical solver, and in practice, it is often a numerical integrator using lower resolution, e.g., much larger time step size, in order to effectively propagate the information forward in time. Ideally, the parareal algorithm can achieve the accuracy of the FP at the efficiency of the CP. The choice of the CP is crucial to achieve the desired speedup on modern  multicore / many-core computer architectures. Nonetheless, the optimal choice of CP remains highly nontrivial. Even worse, for the model \eqref{eqn:pde} with nonsmooth data, the convergence of the parareal algorithm with the standard backward Euler CP is somehow limited \cite{Wu:IMA2015,Mathew:2010,FriedhoffSouthworth:2021,yang2021robust}, which essentially restricts the achievable efficiency, cf. \eqref{eqn:conv-1} below for the precise statement.  This issue occurs with other commonly used CPs \cite[Tables 1 and 2]{FriedhoffSouthworth:2021}. These observations naturally lead to the important question whether it is possible to overcome the convergence barrier of the parareal algorithm by  properly modifying the CP.

In this work we propose an innovative and easy-to-implement framework for designing CPs in parareal algorithms for problem \eqref{eqn:pde}. It builds on the error estimation framework \cite{GanderVandewalle:2007,Southworth:2019,FriedhoffSouthworth:2021,yang2021robust}, and follows the prevalent paradigm of using {optimization} to address challenging tasks in scientific computing. Specifically, based on an analytic convergence factor derived in the error estimation framework, we construct an {optimized coarse propagator (OCP)} to minimize an upper bound on the convergence factor. In this way, the {OCP} is specifically optimized for the underlying FP, which differs markedly from the conventional choice where the CP is almost agnostic to the choice of the FP. We provide a systematic procedure for constructing CPs that respect highly desirable properties, e.g., stability and accuracy. We present a preliminary analysis of the framework, which shows that {OCPs} can indeed achieve the desired goal for linear evolution problems. Several numerical experiments on linear and nonlinear problems, including linear diffusion model, Allen-Cahn model and viscous Burgers equation, are presented to illustrate the approach. The numerical results show that employing optimization techniques consistently enhances the performance compared to the conventional choice of using the backward Euler scheme (or other popular single step methods) as the CP.

Broadly, this work follows the prevalent paradigm of learning/ optimizing numerical algorithms for scientific computing, which combines mathematically driven, handcrafted design of general algorithm structure with a data-driven adaptation to specific classes of tasks \cite{Bar-Sina:2019,GreenfeldKimmel:2019,Mishra:2019,GuoLi:2022,HuangLiXi:2023,IbrahimRuprecht:2023}. For example, Guo et al \cite{GuoLi:2022} present a machine learning approach that automatically learns effective solvers for initial value problems in the form of ordinary differential equations, based on the Runge-Kutta integrator, and learn high-order integrators for targeted families of differential equations. The use of {OCPs} in designing new parareal algorithms has also been explored \cite{YallaEngquist:2018,AgbohDogar:2020,NguyenTsai:2023,IbrahimRuprecht:2023}. Yalla and Engquist \cite{YallaEngquist:2018} explored the use of neural networks to approximate the phase map as a CP for high-dimensional harmonic oscillators and localized multiscale problems. Agboh et al \cite{AgbohDogar:2020} used a feed-forward deep neural network as a CP to integrate an ODE, and observed improved performance. Nguyen and Tsai \cite{NguyenTsai:2023} used supervised learning to enhance its efficiency for wave propagation. Ibrahim et al \cite{IbrahimRuprecht:2023} proposed to use a physics-informed neural network (PINN) as a CP, and show the speedup on the Black-Scholes equation which performs better than a numerical CP. {Very recently, Yoda et al \cite{yoda2024coarse} optimized the coarse-grid operator in multigrid reduction in time (MGRIT) \cite{friedhoff2012multigrid} in order to improve its convergence for solving time-dependent Stokes and Oseen problems.}
The present work continues along this active line of research on enhancing parareal algorithms with learning / {optimizing} concept \cite{YallaEngquist:2018,AgbohDogar:2020,NguyenTsai:2023,IbrahimRuprecht:2023}, but with the major difference of explicitly building analytic insights into the optimization procedure, i.e., stability and consistency, and consequently the {optimized} CPs proposed in this work enjoy rigorous mathematical guarantees. {Moreover, the proposed optimization framework does not use training data, which differs markedly from the existing machine learning techniques based on neural networks.}

The rest of the paper is organized as follows. In Section \ref{sec:prelim}, we recall preliminary materials on single step time stepping methods and parareal algorithm. In Section \ref{sec:TPC} we develop learned coarse propagators, and analyze their properties. Then in Section \ref{sec:experiment} we illustrate the performance of trained coarse propagators on several model problems.
	
\section{Single step methods and parareal algorithm}\label{sec:prelim}
First we describe the abstract framework of single step time stepping methods for solving problem \eqref{eqn:pde}, and give a brief
 overview of the parareal algorithm. For an in-depth treatment of these topics, we refer interested readers to the monograph
 \cite{thomee2007galerkin} and the review \cite{gander201550}.
		
\subsection{Single step solvers for parabolic equations}\label{ssec:single step}
To discretize problem \eqref{eqn:pde} in time, we divide the time interval $(0,T)$ into $N$ equidistant
subintervals, each of length ${\Delta t} = T/N$, and let the grid points be $t_n = n{\Delta t}$, for $n=0,1,\ldots,N$. Then a single step scheme approximates the solution $u(t_{n+1})$ by
\begin{equation}\label{eqn:semi}
	u^{n+1} = r({\Delta t} \A) u^{n} + {\Delta t} \sum_{i=1}^m p_i({\Delta t} \A)
	f(t_n+c_i\Delta t) ,\quad \text{for}~0\le n\le N-1.
\end{equation}
Here $r(\lambda)$ and the sequence $\{p_i(\lambda)\}^m_{i=1}$ are rational functions, and ${c_i}$ are distinct real numbers in $[0,1]$.
Below we impose the following conditions on the scheme \eqref{eqn:semi}:
\begin{itemize}
	\item[{\bf (P1):}] {For any $\lambda > 0$, $|r(\lambda)| < 1$, and $|p_i(\lambda)| < c$ uniformly for all $i = 1, \ldots, m$, where $c$ is a constant.} Further, the numerator of $p_i(\lambda)$
	is of strictly lower degree than the denominator.
	\item[{\bf (P2):}] The scheme \eqref{eqn:semi}  is accurate of order $q$ in the sense that
	$$  r (\lambda) = e^{-\lambda} + \mathcal{O}(\lambda^{q+1}),\quad \text{as}~\lambda\rightarrow0. $$
	Additionally, for every $0\le j\le q$,
	$$ \sum_{i=1}^m c_i^j p_i(\lambda) - \frac{j!}{(-\lambda)^{j+1}}\Big(e^{-\lambda}
	- \sum_{\ell=0}^j \frac{(-\lambda)^\ell}{\ell!}\Big) = \mathcal{O}(\lambda^{q-j}),\quad \text{as}~\lambda\rightarrow0. $$
	\item[{\bf (P3):}] The rational function $r(\lambda)$ is strongly stable in the sense that $|r(\infty)|<1$.
\end{itemize}

See, e.g., the monograph \cite[p. 131]{thomee2007galerkin} for the construction of rational functions satisfying properties {\bf{(P1)}}-{\bf{(P3)}}. {\bf{(P3)}} is crucial for ensuring the convergence of parareal iteration, especially for nonsmooth problem data (e.g., $u^0\in\mathcal{H}$). When $|r(\infty)| = 1$ (e.g., Crank-Nicolson and implicit Runge-Kutta of Gauss type), the convergence of the parareal method depends on two factors. First, the eigenvalues of $\A$ must be bounded from above, which is typically not the case for parabolic equations. Second, the ratio between the coarse and fine step sizes must be sufficiently large, with the precise value determined by the upper bound of the eigenvalues of $\A$; {see Remark \ref{rmk:J} below for further discussions on $J$}. Schemes not satisfying (\textbf{P3}) may lose the optimal convergence rate \cite[Chapter 8]{thomee2007galerkin}.
	
In practice, it is often desirable to choose $p_i(\lambda)$ sharing the denominator with $ r (\lambda)$, i.e., for polynomials $a_i(\lambda)$ and $g(\lambda)$ such that
\begin{equation*}
 r (\lambda) = \frac{a_0(\lambda)}{g(\lambda)}\quad \text{and} \quad p_i(\lambda) = \frac{a_i(\lambda)}{g(\lambda)},\quad \text{for}~i=1,2,\ldots,m.
\end{equation*}
Then the scheme \eqref{eqn:semi} can be written as
\begin{equation*}
g({\Delta t} \A) u^{n+1} = a_0({\Delta t} \A) u^{n} + {\Delta t} \sum_{i=1}^m a_i({\Delta t} \A) f(t_n+c_i\Delta t) ,\quad \text{for}~1 \le n\le N-1.
\end{equation*}

The following error estimate holds for the scheme \eqref{eqn:semi} \cite[Theorems 7.2 and 8.1]{thomee2007galerkin}.
\begin{lemma}\label{lem:conv-00}
Let conditions {\bf{(P1)}}-{\bf{(P3)}} be fulfilled,
$u(t)$ the solution to problem \eqref{eqn:pde}, and $u^n$
the solution to the scheme \eqref{eqn:semi}. Then there holds
		\begin{equation*}%\label{eqn:semi-error}
			\|u^n - u(t_n)\|\leq
			c\, (\Delta t)^q \Big( t_n^{-q} \| u^0 \|
			+ t_n \sum_{\ell=0}^{q-1} \sup_{s\le t_n} \| \A^{q-\ell}f^{(\ell)}(s)  \| + \int_0^{t_n}
			\| f^{(q)}(s) \|\,\d s\Big),
		\end{equation*}
if $u_0\in \H$, $f^{(\ell)}\in C([0,T];\mathrm{Dom}(\A^{q-\ell})$ with $0\le \ell\le q-1$ and $f^{(q)} \in L^1(0,T; \H)$.
\end{lemma}\medskip
	
\begin{remark}\label{rem:conv-00}
By Lemma \ref{lem:conv-00}, under conditions {\bf{(P1)}}-{\bf{(P3)}}, the solution $u^n$ of the scheme \eqref{eqn:semi} converges to the exact one $u(t_n)$ at an order $\mathcal{O}((\Delta t)^q)$, if $f$ and $u_0$ fulfill certain compatibility conditions. For example, when \(\A = -\Delta\), equipped with a zero Dirichlet boundary condition, the requirement is \( (-\Delta)^{\ell} f^{(q-\ell)} = 0 \) on \(\partial\Omega\) for \(0 \le \ell \le q\). To circumvent these stringent conditions, the scheme \eqref{eqn:semi} is assumed to be strictly accurate of order $q$ such that
\[
\sum_{i=1}^m c_i^j p_i(\lambda) - \frac{j!}{(-\lambda)^{j+1}}\Big( r (\lambda)
		- \sum_{\ell=0}^j \frac{(-\lambda)^\ell}{\ell!}\Big) = 0, \quad \text{for} ~0\le j\le q-1.
\]
A scheme is said to be strictly accurate of order $p_0 < p$ if the truncation error vanishes for all $f$ and $u_0$ such that the solution is a polynomial in $t$ of degree at most $p_0-1$.
While a single step method with a given $m \in \mathbb{N}$ might be accurate up to order $2m$ (e.g., Gauss--Legendre method) \cite[Section 2.2]{ehle1969pade}, it can be strictly accurate of {order only up to} \(m+1\) \cite[Lemma 5]{brenner1982single}.
\end{remark}\medskip
	
\begin{remark}\label{rem:conv-01}
The error estimate in Lemma \ref{lem:conv-00} can be improved if the integrator exhibits L-stability, i.e., $r(\infty) = 0$ \cite[Theorem 7.2]{thomee2007galerkin}.
\end{remark}
	
\subsection{Parareal algorithm}\label{ssec:parareal}
Now we present the parareal algorithm for the single step integrator \eqref{eqn:semi}. Let $\Delta T = J \Delta t$ be the coarse step size, $N_c = T/\Delta T\in \mathbb{N}$, and denote $T_n = n\Delta T$. The numerical propagators $\G$ and $\F$ use the coarse and fine time grids, respectively. Typically, $\G$ is an inexpensive, low-order method (e.g., backward Euler scheme), whereas $\F$ is defined by the single step integrator \eqref{eqn:semi}. For any given initial data $v\in \H$ and $f\in C([0,T];\H)$, the coarse and fine propagators are respectively defined by
\begin{align}
	\G(T_n,\Delta T,v,f) &= R({\Delta T} \A) v + {\Delta T} \sum_{i=1}^{M} P_i({\Delta T} \A) f(T_n+C_i\Delta T) , \label{eqn:coarse-integrator}\\
	\F(t_n,\Delta t,v,f) &= r({\Delta t} \A) v + {\Delta t} \sum_{i=1}^{m} p_i({\Delta t} \A) f(t_n+c_i\Delta t) ,\label{eq: integrator}
\end{align}
where $M$ and $m$ denote the numbers of stages of the CP $\mathcal{G}$ and FP $\mathcal{F}$, respectively, and $R$ and $P$ are rational functions for the CP $\mathcal{G}$, and $r$ and $p$ for the FP $\mathcal{F}$. The parareal iteration is given in Algorithm \ref{alg:para}.
	
\begin{algorithm}[ht]
	\center
		\caption{Parareal algorithm.}
		\begin{algorithmic}[1]\label{alg:para}
			\STATE \textbf{Initialization}:
			Compute
			$U_0^{n+1} = \G(T_n,\Delta T,{U_0^{n}}, f)$ with $U_0^0 =u^0$, $n =0,1,...,N_c - 1$;
			\FOR {$k=0,1,\ldots,K$}
			\STATE On each subinterval $[T_n,T_{n+1}]$, sequentially compute for $j=0,1,2,\dots,J-1$
			$$  \widetilde U^{n, j+1}_{k} = \F(T_n+j \Delta t, \Delta t, \widetilde U^{n, j}_{k},f), $$
			with initial value $ \widetilde U^{n, 0}_{k} = U_k^n$, and set $\widetilde U^{n+1}_{k} =  \widetilde U^{n, J}_{k}$.\vskip5pt
			\STATE Perform sequential corrections, i.e., find $U_{k+1}^{n+1}$ by
			$$  U_{k+1}^{n+1} = \G(T_n, \Delta T, U_{k+1}^{n}, f)  + \widetilde U^{n+1}_{k} -   \G(T_n, \Delta T, U_{k}^{n}, f) $$
			with $U_{k+1}^0 = u^0$, for $n =  0,1,...,N_c - 1$;\vskip5pt
			\STATE Check the stopping criterion.
			\ENDFOR
		\end{algorithmic}
\end{algorithm}

The convergence of parallel algorithms has been extensively studied  (see, e.g., \cite{Bal:2005,BrehierWang:2020,GanderVandewalle:2007,Mathew:2010,Wu:IMA2015,WuZhou:2015,FriedhoffSouthworth:2021,yang2021robust} for a rather incomplete list).
For linear evolution equations, Yang et al \cite{yang2021robust} proved that with the backward Euler scheme as the CP, there exists a threshold $J_*>0$,
independent of $\Delta T$, $\Delta t$ and the upper bound of the spectrum of $\A$, such that
under conditions {\bf{(P1)}}-{\bf{(P3)}}, if $J\ge J_*$, then the following bound holds with $\gamma \approx 0.3$,
\begin{equation}\label{eqn:conv-1}
	\max_{1 \le n\le N_c}\|  U_{k}^n - u^{nJ}  \| \le  c \,\gamma^k,
\end{equation}
where $U_k^n$ denotes the parareal solution, and $u^{nJ}$ the fine time stepping solution. Numerically the estimate \eqref{eqn:conv-1} is observed to be sharp. This issue arises with other types of CPs; the work \cite{FriedhoffSouthworth:2021} reported an optimal mesh-independent convergence factor $\gamma \approx 0.26$ when both CP and FP are  the two-stage, second-order singly diagonally implicit Runge  Kutta scheme (SDIRK-22), and $\gamma \approx 0.15$ when both are SDIRK-33, for  $J\in \{4, 8, 16, 32,64\}$. These observations naturally motivate the following question:
\begin{quote} \textit{Are there CPs capable of overcoming these barriers, significantly improving convergence, while maintaining robustness across different $J$ values?}
\end{quote}
We aim at {optimizing} CPs to accelerate the convergence of the parareal algorithm.
% Furthermore, numerically these OCPs also behave robustly for nonlinear problems.

\section{Optimized coarse propagators}\label{sec:TPC}

Now we develop a systematic strategy to construct CPs via optimization, which yields an {Optimized CP (OCP)} that achieves a convergence factor  $\gamma \approx 0.02$ when $J\in [16,128]$ for linear evolution problems, thereby overcoming the barrier \eqref{eqn:conv-1} on the convergence rate. We also discuss basic properties of the resulting parareal algorithm.

\subsection{Error estimation}\label{ssec:error-estimation}
Now we present the error estimation framework in the work \cite{GanderVandewalle:2007,yang2021robust} by incorporating a general coarse/fine propagator. See also a tighter convergence bound in \cite[Theorem 30]{Southworth:2019}.
Consider the coarse and fine propagators defined in \eqref{eqn:coarse-integrator}--\eqref{eq: integrator}, which both satisfy conditions {\bf{(P1)}}-{\bf{(P3)}}. For the difference between the solution $U_{k+1}^{n+1}$ by the parareal time stepping scheme and the fine grid solution $u^{(n+1)J}$,
we deduce
\begin{align*}
	&U_{k+1}^{n+1} - u^{(n+1)J} = R(\Delta T\A)\Big[(U_{k+1}^{n} - u^{nJ})
	- (U_{k}^{n} - u^{nJ})\Big] \\
	& \quad + \F(T_n+(J-1)\Delta t, \Delta t, \widetilde U^{n,J-1}_{k}, f)
	- \F(T_n+(J-1)\Delta t, \Delta t, u^{(n+1)J-1}, f) \\
	=& R(\Delta T\A)\Big[(U_{k+1}^{n} - u^{nJ})
	- (U_{k}^{n} - u^{nJ})\Big]
	+ r(\Delta t \A) (\widetilde U^{n,J-1}_{k}- u^{(n+1)J-1})\\
	=& R(\Delta T\A)\Big[(U_{k+1}^{n} - u^{nJ})
	- (U_{k}^{n} - u^{nJ})\Big]
	+ r(\Delta t \A)^J (\widetilde U^{n,0}_{k}- u^{nJ})\\
	=& R(\Delta T\A)\Big[(U_{k+1}^{n} - u^{nJ})
	- (U_{k}^{n} - u^{nJ})\Big]
	+ r(\Delta t \A)^J ( U_k^{n}- u^{nJ}).
\end{align*}
Upon letting $E_{k}^{n} = U_k^n - u^{nJ}$, we can rewrite the identity by
\begin{equation*}
	E_{k+1}^{n+1} = R(\Delta T\A) (E_{k+1}^n - E_{k}^{n})+r(\Delta t A)^J E_{k}^{n}.
\end{equation*}
Recall that the linear operator $\A$ is positive definite and selfadjoint with a compact inverse
on Hilbert space $(\H, (\cdot , \cdot ))$. By the spectral theory for compact operators \cite{Yosida:1995},
$\A$ has positive eigenvalues $\{\lambda_j\}_{j=1}^\infty$, where $0<\lambda_1\le \lambda_2 \le \ldots $ and $\lambda_j\rightarrow\infty$,
and the corresponding eigenfunctions $\{\phi_j\}_{j=1}^\infty$ form an orthonormal basis of the Hilbert space $\H$.
Then, with $e_{k,j}^n = (E_{k}^{n}, \phi_j)$, by means of spectral decomposition, we derive
\begin{equation*}
	e_{k+1,j}^{n+1} = R(\Delta T\lambda_j)(e_{k+1,j}^n-e_{k,j}^n) +  r(\Delta t \lambda_j)^J e_{k,j}^{n}.
\end{equation*}
Upon setting $d_j = \Delta T \lambda_j$, we get
\begin{equation*}
	e_{k+1,j}^{n+1} =R(d_j) e_{k+1,j}^n +  (r(d_j/J)^J -R(d_j))e_{k,j}^{n}.
\end{equation*}
Applying the recursion repeatedly and noting the condition $ e_{k+1,j}^0=0$ yield
\begin{align*}
e_{k+1,j}^{n+1}&= ( r(d_j/J)^J -R(d_j)) e_{k,j}^n +R(d_j) e_{k+1,j}^n\\
	&= ( r(d_j/J)^J -R(d_j))( e_{k,j}^n+R(d_j) e_{k,j}^{n-1})+  R(d_j)^{2}e_{k+1,j}^{n-1}\\
	&= ( r(d_j/J)^J -R(d_j))( e_{k,j}^n +R(d_j) e_{k,j}^{n-1} + \dots + R(d_j)^{n-1}e_{k,j}^{1}).
\end{align*}
Now, taking the absolute value on both sides gives
\begin{align}\label{eq: error analysis}
		|e_{k+1,j}^{n+1} | &\leq | r(d_j/J)^J -R(d_j)| \cdot (1 +|R(d_j)| + \dots + |R(d_j)|^{n-1})
		\max_{1\le n\le N_c-1}{|e_{k,j}^{n}|} \nonumber \\
		&\leq \frac{|r(d_j/J)^J - R(d_j)| }{1-|R(d_j)|}  \max_{1\le n\le N_c-1}{|e_{k,j}^{n}|}  \nonumber \\
		&\leq \sup_{s \in \{\Delta T\lambda_j\}_{j=1}^{\infty}} \frac{| r(s/J)^J-R(s)| }{1-|R(s)|} \max_{1\le n\le N_c-1}{|e_{k,j}^{n}|}.
\end{align}
This inequality naturally suggests the following componentwise convergence factor
\begin{equation*}
	\kappa (r,R,J,s)=\frac{r(s/J)^{J}-R(s)}{1-|R(s)|},
\end{equation*}
and likewise the convergence factor for any fixed $J_{0}$ is defined by
\begin{equation}\label{eq: Lambda}
	\kappa_{c} (r, R; J_{0})=\sup_{s\in \Lambda} | \kappa (r,R,J_{0},s) |, \quad\mbox{with } \Lambda = \{\Delta T\lambda_j\}_{j=1}^{\infty}.
\end{equation}
Note that $J_0$ is the ratio of the coarse and fine time step sizes that will be used to optimize the CP, cf. Algorithm \ref{algo:coarse}. 
{The componentwise convergence factor is  related to the bound in the pioneer work \cite{dobrev2017two}.}

\begin{remark}\label{rmk:J}
{Throughout this work, we have assumed that the stability function $r(s)$ of the FP is strongly stable, i.e., $|r(\infty)| < 1$. This choice allows great flexibility in the time ratio $J$. In contrast, if $|r(\infty)| = 1$ and $|R(s)| < 1$, then $r(s/J)^J - R(s)$ can be made small for large $s$ only if $J$ is sufficiently large. However, 
a larger 
$J$ implies fewer parallel time slices (and using fewer processors),  limiting the efficiency of parallel computing. Therefore, solvers with $|r(\infty)|=1$, e.g., the Crank--Nicolson scheme, are seldom used as the FPs in parareal algorithms.}
\end{remark}
	
The next result shows that $\kappa_{c} (r, R; J_{0})$ determines the convergence of the parareal method:
If $\kappa_{c}(r, R; J_{0})<1$, then $\max_{1\le n\le N_c} \Vert E_{k}^{n}\Vert \to 0$ as $k \rightarrow \infty$.
That is, the solution $U_k^n$ by the parareal algorithm converges to the fine time stepping solution $u^{nJ}$ as
$k\rightarrow\infty$.

\begin{theorem}\label{thm: estimate for para}
Let conditions {\bf{(P1)}}-{\bf{(P3)}} be fulfilled and the data regularity assumption in Lemma \ref{lem:conv-00} hold.
Let $u^{n}$ be the solution by the FP \eqref{eq: integrator}, and $U^{n}_{k}$ the solution by
Algorithm \ref{alg:para}. Then the following error estimate holds
\begin{equation*}
\max_{1\leq n\leq N_{c}} \Vert U^{n}_{k}-u^{nJ}\Vert  \leq c \kappa_{c} ( r,R;J)  ^{k}.
\end{equation*}
\end{theorem}
\begin{proof}
In Algorithm \ref{alg:para}, the initial guess $U_0^n$ is obtained by the coarse propagator $\mathcal{G}$. Then Lemma \ref{lem:conv-00} (with $q \geq 1$) implies
\begin{align}\label{eqn:conv-3}
	\| U_0^n - u^{nJ} \| &\le \| U_0^n - u(T_n)\|  + \|u(T_n)- u^{nJ} \|  \le c n^{-1}.
\end{align}
With $E_k^n =U_k^n - u^{nJ}$ and $e_{k,j}^n = (E_k^n, \phi_j)$, then the relation \eqref{eq: error analysis} implies
\begin{align*}
\max_{1\le n\le N_c} \|  E_k^n \|^2
& \le  \sum_{j=1}^\infty \max_{1\le n\le N_c} | e_{k,j}^n |^2 \le {\kappa_{c} ( r,R;J) }^2  \sum_{j=1}^\infty \max_{1\le n\le N_c} | e_{k-1,j}^n |^2\\
& \le \cdots     \le \kappa_{c} ( r,R;J) ^{2k}  \sum_{j=1}^\infty \max_{1\le n\le N_c} | e_{0,j}^n |^2.
\end{align*}
This and the inequality $\max_{1\le n\le N_c} | e_{0,j}^n |^2 \le  \sum_{n=1}^{N_c} | e_{0,j}^n |^2$ yield
\begin{align*}
\max_{1\le n\le N_c} \|  E_k^n \|^2  &\le c \kappa_{c} ( r,R;J) ^{2k}  \sum_{j=1}^\infty  \sum_{n=1}^{N_c} | e_{0,j}^n |^2
\le  c \kappa_{c} ( r,R;J) ^{2k}   \sum_{n=1}^{N_c} \| E_{0}^n \|^2\\
&\le  c^{2} \kappa_{c} ( r,R;J) ^{2k}   \sum_{n=1}^{N_c} n^{-2}
\le c^{2} \kappa_{c} ( r,R;J) ^{2k} ,
\end{align*}
where the second last inequality is due to the estimate \eqref{eqn:conv-3}.
\end{proof}

\subsection{{Optimizing} coarse propagator}
Given a high-order single step scheme as the FP $\mathcal{F}$,  we shall learn a CP
to expedite the convergence of parareal iteration. We consider a parametric stability function $R(\lambda,\bt)$
in terms of $\lambda$, defined by
\begin{equation}\label{eqn:stability-parametric}
	R( \lambda,\bt)  =\frac{P_{n}(\lambda,\bt)  }{Q_{m}( \lambda,\bt)  } =\frac{\sum^{n}_{i=0} a_{i}\lambda^{i} }{1+\sum^{m}_{i=1} e^{b_{i}} \lambda^{i} }.
\end{equation}
The parameters is grouped as $\bt=( b_{1},...,b_q, b_{q+1},...,b_m,a_{q+1},...,a_n)$, %and $\bt'=$, 
{with $q < n \leq m$. Note that $( a_0,...,a_q)$ is excluded from the vector $\bt$.} We discuss this exclusion and the parameter $q$ below. The exponential representation in the denominator of $R(\lambda,\t)$ is to enhance the training stability. For any optimal $\bt^*$, according to conditions {\bf{(P1)}}-{\bf{(P3)}}, ideally, the {optimized} stability function $R( \lambda,\bt^*)$ in \eqref{eqn:stability-parametric} of the OCP $\mathcal{G}_{\bt^*}$ should satisfy \begin{enumerate}
\item[(i)] {\bf (P1), (P3)}: The OCP $\mathcal{G}_{\bt^*}$ is stable and strongly stable, i.e.,
\begin{equation*}
|  R(  \lambda,\bt^*)  | < 1 ~\text{for}~ \lambda > 0 ~\text{and}~ |R(  \infty,\bt^*) | < 1 .
\end{equation*}
\item[(ii)] {\bf (P2)}: The OCP $\mathcal{G}_{\bt^*}$ has an accuracy of order $q$, i.e., $R(\lambda,\bt^*)  = e^{-\lambda} + \mathcal{O}(\lambda^{q+1})$.
\end{enumerate}

Inspired by these requirements, we define the error function $e(\lambda,\bt^*)$ by
\begin{equation*}
	e( \lambda,\bt^*)  =e^{-\lambda } -R(\lambda,\bt^*)  =e^{-\lambda }-\frac{P_{n}( \lambda,\bt^*)  }{Q_{m}( \lambda,\bt^*)  } ,
\end{equation*}
and let $E( \lambda,\bt^*)  \coloneqq e( \lambda,\bt^*)  Q_{m}( \lambda ,\bt^*)  =e^{-\lambda }Q_{m}( \lambda,\bt^*)  -P_{n}(\lambda,\bt^*)$.
Given $q < n$,  {we request $E^{(k)}(0,\bt^*) = 0,\  k=0,1,...,q$, i.e.,}
\begin{equation*}
	\frac{{\rm d}^{k}}{{\rm d}\lambda^{k} }( e^{-\lambda }Q_{m}) ( 0,\bt^*)  =\frac{{\rm d}^{k}}{{\rm d}\lambda^{k} } P_{n}( 0, \bt^*),  \quad  k=0,1,\cdot \cdot \cdot ,q.
\end{equation*}
By Taylor expansion, with $e^{-\lambda }=\sum^{\infty }_{i=0} c_{i}\lambda^{i}$, we obtain
\begin{equation}\label{eq:determine a_1}
	\sum^{k}_{j=0} c_{j}e^{b_{k-j}}=a_{k},\quad  k=0,1,\cdot \cdot \cdot ,q.
\end{equation}
That is, the parameters $a_k$, $k=0,\ldots,q$, is fully determined by $(b_1,\ldots,b_q)$, one block of $\bt$, under the consistency constraint, and it suffices to optimize only the parameters $\bt$.
	
The optimized stability function $R(\lambda,\bt^*)$ is chosen to minimize the convergence factor $\kappa_{c}(r,R;J_{0})$. However, it is inconvenient to compute $\Lambda$ since $\lambda_{j}\rightarrow \infty$. So we approximate the operator $\A$ with a matrix $A_h$ \cite[Section 1]{thomee2007galerkin}, and approximate the spectrum $\Lambda$ in \eqref{eq: Lambda} with the discrete spectrum $\Lambda_h$ (associated with $A_h$). In practice, this requires deriving upper and lower eigenvalue bounds of $A_h$, e.g., using Rayleigh quotient \cite{Golub:2013}. To approximate the set $\Lambda_h$, we sample uniformly {$N_\Lambda$ points} over the eigenvalue range of $\Lambda_h$. %\red{Specifically, if $\A = -\Delta$, we can exactly compute the eigenvalues of $A_h$ and use the complete set of eigenvalues as $\Lambda_h$, rather than sampling.}
	
Now consider the {optimization} of a CP $\mathcal{G}_{\bt}$. Given the stability function $r(\lambda)$ of the FP $\mathcal{F}$,
we select the parameters $m$, $n$, $q$, $J_{0}$ and $N_\Lambda$. The resulting optimization problem reads
\begin{align}\label{eq:opt problem}
\underset{\bt}{\text{min}}\  \kappa_{c}(r,R;& J_0)\quad
\text{s.t.} \quad  \sum^{k}_{j=0} c_{j}e^{b_{k-j}}= a_{k}, \, k = 0,1,\dots,q, \, \mbox{and}\,
|R(\lambda,\bt)| < 1, ~ \lambda > 0.
\end{align}
These constraints explicitly enforce properties (\textbf{P1})--(\textbf{P3}), i.e., stability and convergence order. The construction naturally respects highly desirable features of an ideal CP.

To strictly enforce the condition $ |R(\lambda,\bt)| < 1$ for $\lambda > 0$, we define a set $\Lambda_{R}$ (with a cardinality $N_R$), which
encompasses all positive critical
points of $R(s,\bt)$ in $s$, for any fixed $\bt$. Indeed, the extremum of $R$ is achieved at critical points. Hence, it suffices to enforce the following conditions
\begin{equation*}
|R(s,\boldsymbol{\theta})|, |R(\infty,\boldsymbol{\theta})| < 1, \quad \text{for all } s \in \Lambda_R.
\end{equation*} Since the stability function $R(s,\bt)$ is known,
computing $\Lambda_R$ \textit{a priori} is feasible for a fixed $\bt$. Then it suffices to verify the condition $ |R(\lambda,\bt)| < 1$ for $\lambda \in \Lambda_{R}$. To illustrate this, consider the choice $m=n=2$ and $q=1$. Using the consistency condition, the parametric stability function $R\left( \lambda ,\bt\right)$ is given by 
\begin{equation*}
    R\left( \lambda ,\bt\right)  =\frac{1+\left( e^{b_{1}}-1\right)  \lambda +a_{2}\lambda^{2} }{1+e^{b_{1}}\lambda +e^{b_{2}}\lambda^{2} }.
\end{equation*}
Then using the first derivative 
\begin{equation*}
\frac{\d R(\lambda,\bt)}{\d\lambda} = \frac{ \lambda^2 (a_2 e^{b_1} + e^{b_2} - e^{b_1 + b_2}) - 2 \lambda ( e^{b_2} - a_2) - 1}{(\lambda^2 e^{b_2} + \lambda e^{b_1} + 1)^2},
\end{equation*}
we can determine the critical points $\lambda_1$ and $\lambda_2$ in closed form, and obtain the set $\Lambda_R = \{\max\{\lambda_1,0\},\max\{\lambda_2,0\}\}$. In practice, the integers $m$ and $n$ are small, in order to maintain low computational expense.
The choice of $\Lambda_h$ and $\Lambda_R$ is sufficient to ensure that conditions \textbf{(P1)} and \textbf{(P3)} hold.

We adopt the barrier method \cite[Chapter 13]{luenberger1984linear}, in order to strictly enforce the condition:
\begin{equation}\label{eq:loss_2}
	\mathcal{L}_{\rm b}(\boldsymbol{\theta}) = \frac{1}{N_\Lambda+N_R} \sum_{s \in \Lambda_h\cup\Lambda_R} \log(1 - R^2(s, \boldsymbol{\theta})) + {\log\left(1 - \left(\frac{a_n}{e^{b_m}}\right)^2\right).}
\end{equation}

Then the total loss $\mathcal{L}_{\rho}$ is given by
\begin{equation*}
    \mathcal{L}_\rho (\bt)= \mathcal{L}_{\rm s}(\bt)-\rho \mathcal{L}_{\rm b}(\bt),\quad \mbox{with }\mathcal{L}_{\rm s}(\bt) = \sup_{s\in \Lambda_h} \frac{ |r(s/J_{0})^{J_{0}}-R(s,\bt)|}{1-|R(s,\bt)|},
\end{equation*}
where $\rho>0$ is the weight for the barrier function. To ensure the convergence of the barrier method for problem \eqref{eq:opt problem} during optimization, we choose the weight $\rho$ by a path-following strategy: we initialize with $\rho_0$, and then geometrically decrease its value, i.e., $\rho_{i+1}=\beta\rho_i$, for some $\beta\in(0,1)$. Let $\theta_k^*$ be a minimizer of the loss $\mathcal{L}_{\rho_k}(\theta)$. Then any limit point of the sequence $\{\bt_k^*\}$ solves problem \eqref{eq:opt problem} \cite[Chapter 13]{luenberger1984linear}. The complete procedure for optimizing the CP $\mathcal{G}_{\bt}$ is given in Algorithm \ref{algo:coarse}. To minimize the loss $\mathcal{L}_\rho$, we employ subgradient descent; see \cite{goffin1977convergence, bianchi2022convergence, scaman2022convergence, khaled2020better} for relevant convergence analysis.

\begin{algorithm}[hbt!]
\caption{Optimizing the CP for the FP $\mathcal{F}$.}
\begin{algorithmic}[1]
\STATE Specify parameters $m,n,q,J_{0},N_\Lambda,$ $\rho_0$, $\beta$, and initialize $\bt$ randomly.
\STATE Compute the set $\Lambda_h$.
\FOR {$i=0,1,\ldots$}
\STATE $\rho_{i+1}=  \beta \rho_{i}.$
\FOR {$k=0,1,\ldots,K$}
\STATE Given $\bt$, find the set $\Lambda_{R}$.
\STATE Construct the loss  $\L_{\rho_i}(\bt) = \L_{\rm s}(\bt) - \rho_{i} \cdot \L_{\rm b}(\bt)$\red{.}
\STATE Minimize the loss $\mathcal{L}_{\rho_i}(\bt)$  and update $\bt$.
%\STATE \sout{Check the stopping criterion.}
\ENDFOR
\STATE Check the stopping criterion.
	\ENDFOR
	\RETURN $\bt$.
\end{algorithmic}
\label{algo:coarse}
\end{algorithm}

In practice, a decaying step size schedule and a large $\beta$ are adopted for the optimizer. The empirical convergence to a local minimum is indicated when the norm of the loss gradient with respect to the parameters falls below a given tolerance within the loop over $i$ (which serves as the stopping criterion at step 10) and then the iteration is terminated. Numerically, other off-shelf optimizers, e.g., Adam and L-BFGS, can yield similar results.

Next we compute $P_i(\lambda,\bt)$, and request that the OCP is strictly accurate of order $q$:
\begin{equation}\label{eq:accurate}
	\sum_{i=1}^m C_i^j P_i(\lambda,\bt) - \frac{j!}{(-\lambda)^{j+1}}\Big( R (\lambda,\bt)
		- \sum_{\ell=0}^j \frac{(\lambda)^\ell}{\ell!}\Big) =0,\quad \forall  0\le j\le q-1.
\end{equation}
We employ a uniform distribution for the values of $C_i^j$ within the interval $[0, 1]$, so as to solve
$P_i(\lambda,\bt)$ via linear systems. When $i=1$, we take $C_{i}^j=1$ in order to have the better stability of a backward scheme. Note that the functions $P_i$ do not play a role in the convergence analysis of OCP or the convergence factor $\kappa_c$ for linear evolution problems. For general nonlinear problems, one may also include \eqref{eq:accurate} as a side constraint in the optimization formulation in order to further enhance the performance, which however will not be pursued below. Finally, in view of the standard CP update \eqref{eqn:coarse-integrator}, $R ( \lambda ,\bt ) $ and $P_i ( \lambda,\bt  ) $ together define the OCP:
\begin{equation*}
	\G(T_n,\Delta T,v,f) = R({\Delta T} A,\bt) v + {\Delta T} \sum_{i=1}^{M_{1}} P_i({\Delta T} A,\bt)f(T_n+C_i\Delta T).
\end{equation*}
In practice, we take two, three, four-stage Lobatto IIIC method and three-stage Radau IIA method
\cite{wanner1996solving,HAIRER199993} as the FP $\mathcal{F}$, and give the learned $R(\lambda , \bt)$ in Appendix \ref{app:coarse}.
\begin{proposition}
For any optimal choice $\bt^*$ of the parameter $\bt$, the OCP is consistent and has an accuracy of order $q$. Furthermore, it is strongly stable.
\end{proposition}
\begin{proof}
Note that we learn the OCP according to conditions {\bf{(P1)-(P3)}}.
First, the OCP satisfies condition {\bf{(P2)}} and equation \eqref{eq:accurate}, so it is consistent, with at least an accuracy of order $q$.
Second, due to the presence of the penalty term $\L_{\rm b} (\bt)$ \eqref{eq:loss_2}, the OCP satisfies both conditions {\bf{(P1)}} and {\bf{(P3)}}, implying its strong stability.
\end{proof}

\begin{remark}
In exact arithmetic, the parareal algorithm with the OCP converges to the fine grid solution. In practice, the round-off error can be significant, due to the presence of the discrete operator $A_h^2$ in the numerator of the stability function $R(\lambda,\theta)$. To mitigate the stability issue, we adopt the following reformulation
\begin{equation*}
R\left( \lambda,\theta \right) =\frac{a_{0}+a_{1}\lambda +a_{2}\lambda^{2} }{1+e^{b_{1}}\lambda +e^{b_{2}}\lambda^{2} } =\frac{a_{2}}{e^{b_{2}}} +\frac{\left( a_{0}e^{b_{2}}-a_{2} \right) +(a_{1}e^{b_{2}}-a_{2}e^{b_{1}})\lambda }{(1+e^{b_{1}}\lambda +e^{b_{2}}\lambda^{2} )e^{b_{2}}}.
\end{equation*}
Numerically, it allows achieving an accuracy close to machine precision.
\end{remark}
	
\subsection{Robustness with respect to the ratio $J$}\label{subsec:sensitivity}
In Algorithm \ref{algo:coarse}, the ratio $J$ of the coarse step size $\Delta T$ to the fine one $\Delta t$ is predetermined to $J_0$ when {optimizing} the parameters. It is important to study the robustness
of the resulting algorithm with respect to the choice of $J$. The next result shows that the convergence factor remains stable in $J$.
To this end, we study the supremum of $\kappa(r,R,J,s)$ over the interval $\mathbb{R}_+:=(0,\infty)$:
$\kappa_{c}(r,R; J) \leq	\sup_{s\in \mathbb{R}_+} | \kappa ( r,R,J,s)  | $. The convexity assumption $g''(J)\geq0$ on $g(J)=r(s/J)^J$ with respect to $J$ (for $s>0$) holds for two and four-stage Lobatto IIIC methods, and three-stage Radau IIA method. In practice, the ratio $J$ is important: a too small $J$ leads to big and non-negligible cost, whereas a too large $J$ requires more outer loops. For $J$ values smaller than 16, numerically one observes that the convergence factor $\kappa(r,R,J,s)$ is also small; see Table \ref{tab:compare} for concrete numbers. Hence, we have chosen $J \in [ 16,128 ]$ in the discussion below. 
\begin{theorem}\label{thm:sensitivity of J}
Let $J_{1},J_{2}\in [ 16,128]$ such that $\  J_{1}\leq J_{2}$. Let $r(s)$ be the stability function of the FP, and satisfy the following three conditions:
{\rm(i)} conditions {\bf{(P1)-(P3)}}; {\rm(ii)}  $g''( J)  >0$, with $g(J)=r( s/J)^{J}$  for any given $s\in \mathbb{R}_+$; and {\rm(iii)} $r( s)  \geq 0$, for any $s\in \mathbb{R}_+$.
Let $R(s,\bt)$ be the stability function of the OCP. Then there holds
\begin{equation*}
\sup_{s\in \mathbb{R}_+} | \kappa ( r,R,J_{2},s)|- \sup_{s\in\mathbb{R}_+} | \kappa ( r,R,J_{1},s)|\leq 112\sup_{s\in \mathbb{R}_+}|h(s)|,
\end{equation*}
with the function $h(s)$ given by
\begin{equation}\label{eqn:h}
  h(s) =\frac{r( s)^{15}}{1-|R( 16s,\bt)|} \cdot(r( s)\ln  r(s)- r^{\prime }( s)s ).
\end{equation}
\end{theorem}
\begin{proof}
By definition, for any $s^*\in \mathbb{R}_+$, we have
\begin{align*}
	 &|\kappa(r, R, J_{2},s^*)| - \sup_{s\in \mathbb{R}_+} |\kappa(r, R, J_{1},s)| \\
\leq& |\kappa(r, R, J_{2},s^*)| - |\kappa(r, R, J_{1},s^*)| \leq |\kappa(r, R, J_{2},s^*) - \kappa(r, R, J_{1},s^*)|.
\end{align*}
 We claim that for any $J\in [ 16,128]$ and $s>0$
\begin{equation}\label{eq:kappa}
	\frac{\rm d}{{\rm d}J}  \kappa (r,R,J,s ) \leq 0 \quad \text{and}\quad \frac{{\rm d}^{2}}{{\rm d}J^{2}}  \kappa ( r,R,J,s ) \geq 0.
\end{equation}
Now with $g(J)$, we have
$\kappa ( r,R,J,s )  =\frac{g( J)  -R( s,\bt )  }{1-|R( s,\bt )|  }$.
Since $| R( s,\bt)|<1$ for $s>0$ (and independent of $J$), it suffices to prove $g'(J) \leq 0 $ and $ g''(J) \geq 0$ for $J\in [ 16,128]$ and $s>0$. Let
$f(J) := \ln g( J)  =J \ln r( s/J )$.
Then we have
\begin{align}\label{eq:g''}
g'(J) =f'(J)  g(J) \quad\mbox{and}\quad
	g''(J)  =\frac{f''( J)  g(J)^2  +( g'(J) )^{2}  }{g( J)}.
\end{align}
The first two derivatives of $f(J)$ in $J$ are given by
\begin{align}\label{eq:f''}
	f'( J) &=\ln r({s}/{J} ) -\frac{s}{r( s/J) } \frac{r'( s/J) }{J}\mbox{ and }
	f''( J) =\frac{r''( s/J) r( s/J) -( r'( s/J) )^{2} }{r(s/J)^{2}J^{3}} s^{2}.
\end{align}
 Clearly, $r( 0)  =1$ and $r'( 0)  = -1$, so we have $f'(\infty )  =0$.
By assumption (ii), we have  $g''(J)  \geq 0$, and $g'(\infty) = 0$. These facts imply $g'(J) \leq 0$, yielding the estimates in \eqref{eq:kappa}. Since $\kappa ( r,R,J, s_{2}) $ is twice differentiable in $J$ over the interval $[ 16,128]$, we obtain
\begin{align*}
	& |\kappa(r, R, J_{2},s^*) - \kappa(r, R, J_{1},s^*)|
	\leq \sup_{J \in [16, 128]}\Big|\frac{\rm d}{{\rm d}J} \kappa(r, R, J,s^*)\Big| (J_{2}-J_{1}) \\
	\leq & \Big|\frac{\rm d}{{\rm d}J} \kappa(r, R, 16,s^*)\Big| (128-16)
=112\sup_{s\in \mathbb{R}_+} \Big|\frac{\rm d}{{\rm d}J} \kappa(r, R, 16,16s)\Big|.
\end{align*}
Then noting the identity $h( s)=\frac{\d}{\d J} \kappa(r, R, 16,16s)=\frac{r(s)^{15} }{1-|R( 16s,\bt) | } \cdot(r(s)  \ln  r(s)-r'(s)s) $, and
taking the supremum over all $s\in\mathbb{R}_+$ complete the proof.
\end{proof}

The next result discusses the robustness of OCP for the two-stage Lobatto IIIC method.	
\begin{corollary}\label{coro:upper bound for 2 stage Lobatto}
The two-stage Lobatto IIIC method satisfies the assumptions in
Theorem \ref{thm:sensitivity of J}, and $|h(s)|\leq \text{4.91e-4}$. Further, for $J_{1},J_{2}\in [ 16,128]$ with $J_{1}\leq J_{2}$, we have
\begin{equation*}
\sup_{s_{2}\in\mathbb{R}_+} | \kappa( r,R,J_{2},s_{2})|-\sup_{s_{1}\in \mathbb{R}_+}| \kappa ( r,R,J_{1},s_{1})|  \leq 0.055.
\end{equation*}
\end{corollary}
\begin{proof}
For the two-stage Lobatto IIIC method, we have $r(s)  =\frac{2}{s^{2}+2s+2}$.
We claim that $r(s) $ satisfies the assumptions in Theorem \ref{thm:sensitivity of J}. Indeed, Assumptions (i)
and (iii) hold trivially. To see Assumption (ii), we write $r(s)  ={P_{r}(s)}/{Q_{r}(s)}$ and
employ equations \eqref{eq:g''} and \eqref{eq:f''}. We claim the following inequality
\begin{equation}\label{eqn:convex}
    r''(s) r(s) \geq (r'(s))^2,\quad\forall s\in \mathbb{R}_+,
\end{equation}
or equivalently,
\begin{equation*}
	\frac{P_{r}( s)  P''_{r}( s)  -(P'_{r}( s)  )^{2}}{P_{r}( s)^{2}  } \geq \frac{Q_{r}( s)  Q''_{r}( s)  -(Q'_{r}( s)  )^{2}}{Q_{r}( s)^{2}  } ,\quad  \forall s\in \mathbb{R}_+,
\end{equation*}
which can be verified directly using the explicit form of $r(s)$. For the two-stage Lobatto IIIC method, the stability function $R(s,\bt)$ of the OCP is given by (cf. Appendix \ref{app:coarse})
\begin{equation*}
R(s,\bt) = \frac{{1.0 - 0.20967s + 0.00484s^2}}{{1.0 + 0.79033s + 0.37931s^2}} ,
\end{equation*}
For the function $h$ defined in \eqref{eqn:h}, we claim $|h(s)| \leq {4.91\cdot10^{-4}}$ for $s>0$. We divide the proof of the desired assertion into the following three steps: (i) $h(s)\leq 0$ for $s>0$; (ii) bound $h$ for $s\in [1/2,\infty)$,  i.e., $h(s) >{-2.1\cdot10^{-4}}$ for $s\in [ 1/2 ,\infty ) $; and (iii) bound $h$ for $s\in (0,1/2)$, i.e., $h(s)\geq 0.506\cdot ({-9.5\cdot10^{-4}})  \geq{-4.91\cdot10^{-4}}$.\\
\textbf{Step (i)}: Let
$\varphi(s)=\ln r(s)-\frac{r'(s)s}{r(s)}$. From the inequality \eqref{eqn:convex}, we deduce
$\varphi'( s)  =\frac{s((r'(s))^{2}  -r''(s) r(s))}{r^{2}(s)} \leq 0$. Then we derive $\varphi ( s)  \leq \varphi (0^+)  =0$, which implies $h( s)  \leq h( 0^+)  \leq 0$.\\
\textbf{Step (ii)}: For $s\in [1/2,\infty )$, we have
\begin{equation*}
    \frac{\rm d }{{\rm d} s} R(s,\bt) \approx \frac{1.0-0.74894s+0.083355s^{2}}{(1.0+0.79033s+0.37931s^{2})^2} .
\end{equation*}
Thus $R(16s,\bt)$ has a unique extremum at $\bar s \approx 0.6353$, and it is decreasing on $[1/2,\bar s]$ and increasing on $(\bar s,\infty)$. Hence, $|R(16s,\bt)|\leq \max\{ |R(16\bar s,\bt)|,|R(\infty,\bt)|\}\leq {1.4\cdot10^{-2}}$, and
\begin{equation*}
	h( s)  \geq \frac{r(s)^{16}\ln r( s)  }{1-|R( 16s,\bt )|  } \geq 1.015 \cdot r(s)^{16}\ln r( s) := \psi(s).
\end{equation*}
Clearly, {$\psi'(s)=1.015r(s)^{15} r'(s)(16\ln  r(s)+{1})$}, and also  $r>0$ and $r'<0$, from which it follows that $\psi'(s)>0$ for $s>1/2$. Hence, we deduce
$h( s)  \geq \psi ( s)  \geq \psi(1/2)  >{-2.1\cdot10^{-4}}$.\\
\textbf{Step (iii)}: For $s\in (0,1/2]$, note that there exists a unique root $s^{\ast}\approx 0.341$ of $R(16s^{\ast},\bt)$. Let $\xi ( s)  =\frac{s}{1-|R( 16s,\bt ) | }$ and $\eta ( s)  =\frac{\ln  r( s)    }{s}$. Then
\begin{align*}
	\xi'(s)&=\begin{cases}\frac{1-R(16s,\bt)+16R'(16s,\bt) s}{(1-R(16s,\bt))^{2}} , &0<s<s^{\ast }_{1},\\
    \frac{1+R( 16s,\bt)  -16R'(16s,\bt)s}{(1+R(16s,\bt))^{2}} ,&s^{\ast }_{1}\leq s<0.5,
    \end{cases} ~\mbox{and}~
	\eta'(s) =\frac{\frac{r'(s)s}{r(s)}-\ln r(s)}{s^{2}} =-\frac{\varphi(s)}{s^{2}} .
\end{align*}
Hence, $\eta ( s)\leq 0$ and $ \xi'( s)  ,\eta' ( s)  \geq 0$ when $s\in ( 0,1/2) $.  Then $h(s)$ can be bounded as
\begin{align*}
	h( s)  &=\frac{s}{1-|R( 16s,\bt)|  } \cdot \Big( r(s)^{16}\frac{\ln  r( s)}{s} -r(s)^{15}r'(s)\Big) \\
	&=\xi( s)  \cdot ( r( s)^{16}  \eta (s)-r(s)^{15}r'( s)  ) \\
	& \geq \xi(1/2)  \cdot ( r(s) ^{16} \eta ( 0^+ )  -r(s)^{15}r'( s)  )  \\
	&\geq 0.506( -r( s)^{16}-r( s)^{15}r'(s)  ),
\end{align*}
Further, let $\phi ( s)  =r( s)^{16}  +r( s)^{15} r'( s)$. Since
$\phi'( s)  = r(s)^{14}(16r( s)r'( s)  + 15(r'(s))^{2}+r''( s)r(s))$, there is a unique root $\bar s^{\ast}$ of $\phi'( s) $ when $s\in (0,1/2) $, and that $\phi(\bar s^{\ast })\leq {9.5\cdot10^{-4}}$, i.e., $h(s)\geq 0.506\cdot({-9.5\cdot10^{-4}})  \geq{-4.91\cdot10^{-4}}$.\\
These three steps and Theorem \ref{thm:sensitivity of J} yield
$\sup_{s\in\mathbb{R}_+} | \kappa ( r,R,J_{2},s)  |  -\sup_{s\in \mathbb{R}_+ } | \kappa ( r,R,J_{1},s)  |  \leq 0.055$.
This completes the proof of the corollary.
\end{proof}
	
\begin{remark}
Similar to Corollary \ref{coro:upper bound for 2 stage Lobatto}, one can also derive upper bounds for four-stage Lobatto IIIC and three-stage Radau IIA methods. The stability function of the three-stage Lobatto IIIC method does not satisfy assumption (iii) in Theorem \ref{thm:sensitivity of J}.
\begin{figure}[htbp]
	\centering
 \setlength{\tabcolsep}{2pt}
\begin{tabular}{cc}
    \includegraphics[width=.46\textwidth,trim={2cm 3cm 2cm 3cm},clip]{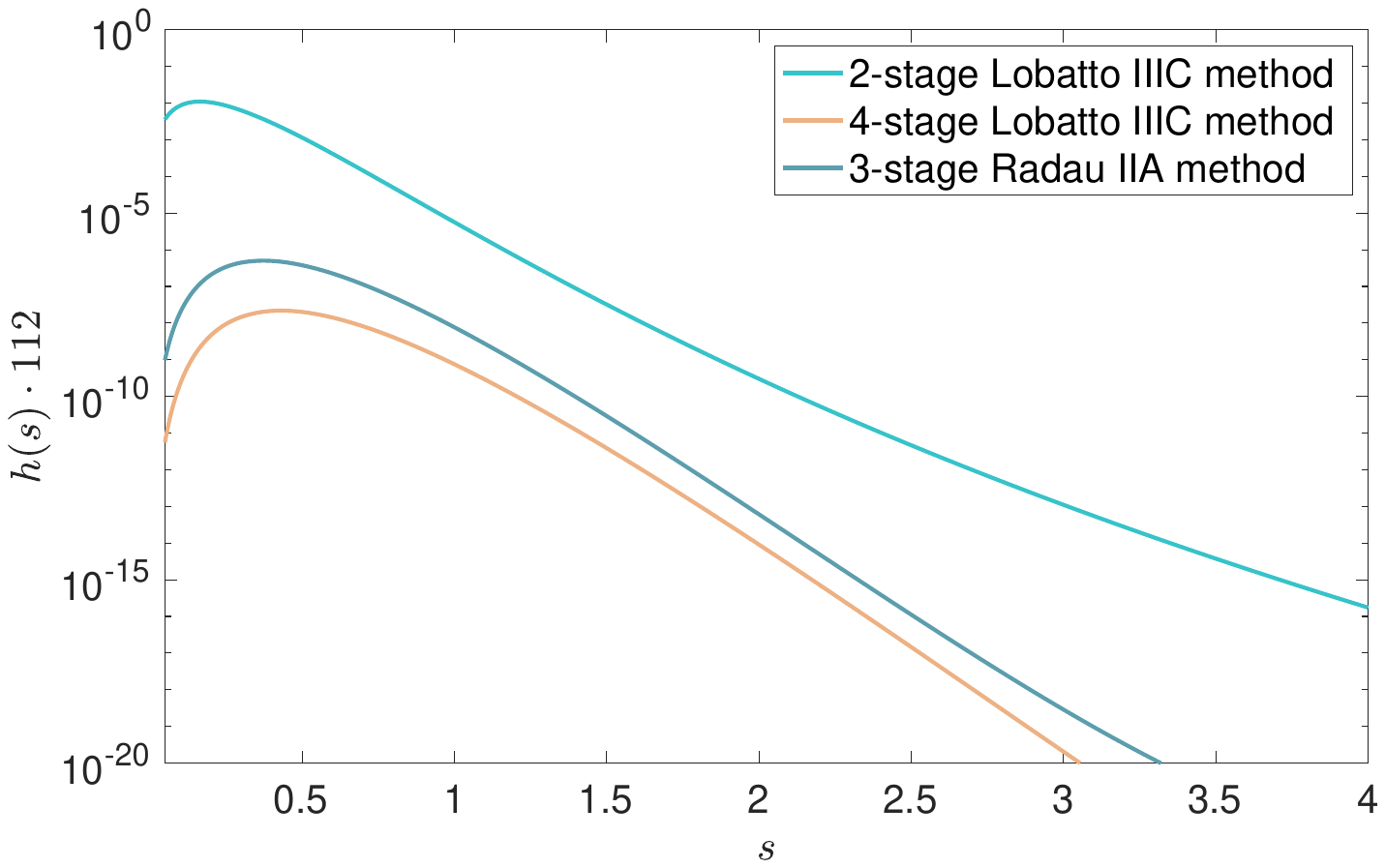}     &  \includegraphics[width=.49\linewidth,trim={2cm 3cm 2cm 3cm},clip]{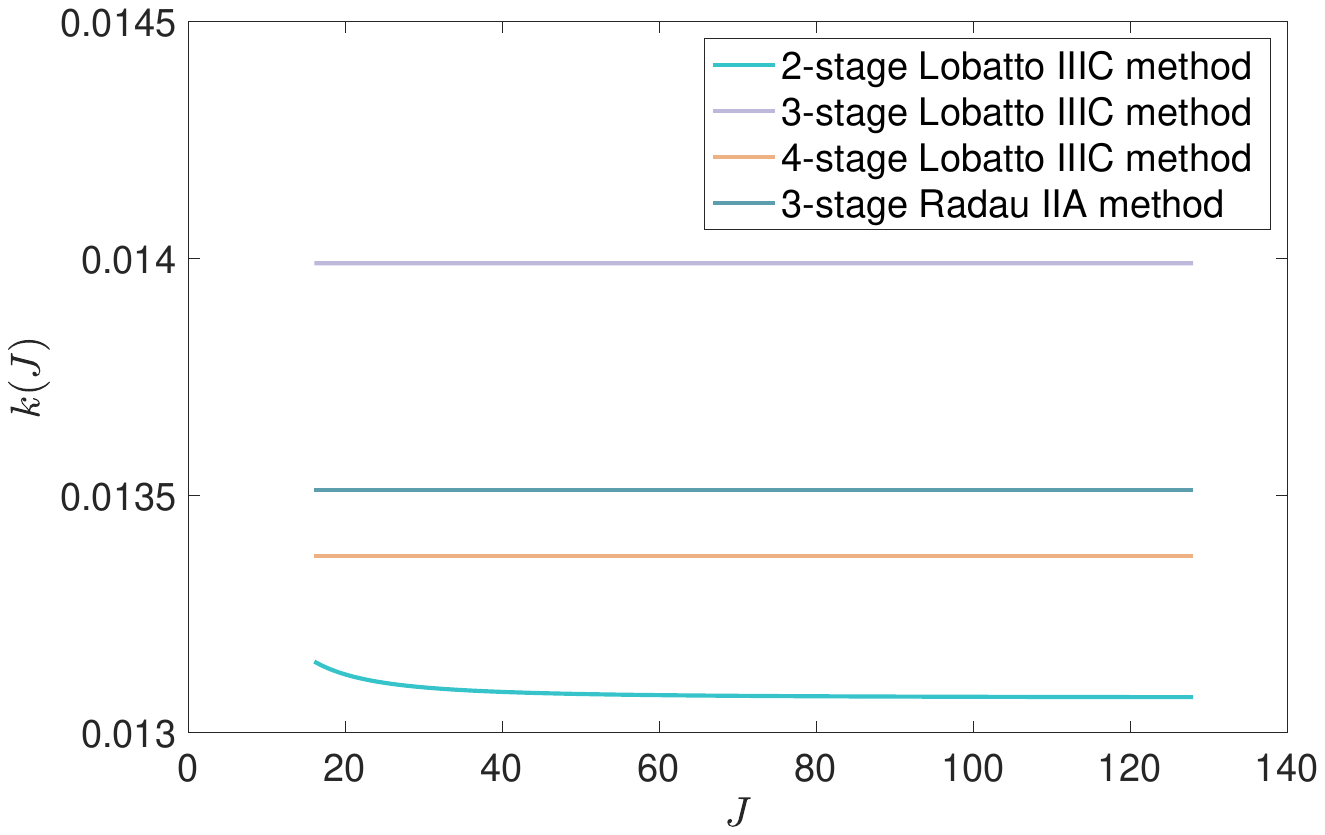}\\
     \begin{minipage}[c]{0.48\textwidth}
		\centering
		\begin{tabular}{cc}
			\toprule
			FP &   ${\sup\limits_{s \in\mathbb{R}_+} |h(s)|}\cdot 112$
 \\
			\midrule
			2-stage Lobatto IIIC& {$1.10\cdot10^{-2}$} \\
			4-stage Lobatto IIIC& {$2.20\cdot10^{-8}$}\\
			3-stage Radau IIA   & {$5.20\cdot10^{-7}$}\\
			\bottomrule
		\end{tabular}
	\end{minipage}    & \begin{minipage}[c]{0.48\textwidth}
		\centering
		\begin{tabular}{cc}
			\toprule
			FP & ${\sup\limits_{J\in [16,128]} k(J)}$  \\
			\midrule
			2-stage Lobatto IIIC& {$1.35\cdot10^{-2}$} \\
			3-stage Lobatto IIIC& {$1.40\cdot10^{-2}$} \\
			4-stage Lobatto IIIC& {$1.35\cdot10^{-2}$} \\
			3-stage Radau IIA   & {$1.36\cdot10^{-2}$} \\
			\bottomrule
		\end{tabular}
	\end{minipage}
\\
  (a) upper bound $\cdot 112$ & (b) uniform bound
\end{tabular}
 	\begin{minipage}[c]{0.48\textwidth}
		\centering
		
	\end{minipage}
	\hfill
	\caption{Numerical bounds for the propagators, and convergence rate versus $J$.}
	\label{fig:lower bound}
\end{figure}

\end{remark}

{Next we numerically illustrate Corollary \ref{coro:upper bound for 2 stage Lobatto} and Theorem \ref{thm:sensitivity of J} in Fig. \ref{fig:lower bound}. Fig. \ref{fig:lower bound}(a) plots $112\cdot h(s)$ (with $h(s)$ defined in \eqref{eqn:h}) for the stability functions of two- and four-stage Lobatto IIC method and three-stage Radau IIA method, whereas the table below reports the bound $112\cdot \sup_{s\in\mathbb{R}_+}|h(s)|$, which represent an upper bound on the sensitivity. For $ s \in [4, \infty) $, the value of $ |h(s)| \cdot 112 $ is always less than $1.0 \cdot 10^{-15} $, and thus not shown in the figure. The results suggest that the upper bound in Corollary \ref{coro:upper bound for 2 stage Lobatto} tends to be slightly conservative. In Theorem \ref{thm:sensitivity of J}, we use the factor $h(s)$ to bound the difference $k(J_2)-k(J_1)$, with $k(J)= \sup_{s \in\mathbb{R}_+} | \kappa (r,R,J,s)|$. In Fig. \ref{fig:lower bound}(b), we plot $k(J)$ for four methods, i.e., two, three, four-stage Lobatto IIIC and three-stage Radau IIA. (Note that $k(J)$ is only defined for the three-stage Lobatto IIIC method for $J\in\mathbb{N}$). The figure indicates that the variation of $k(J)$ with respect $J$ is very small for all the methods, clearly indicating the robustness of the convergence factor $k(J)$ in $J$. The uniform bound $\sup_{J\in [16,128]} k(J)$ is consistently small (at a level $1.40\times10^{-2}$) for different FPs.}

%\begin{remark}\label{rmk:insensitive to J}
%In Theorem \ref{thm:sensitivity of J}, we have derived an upper bound, which however is not sharp \red{because we construct $h(s)$ to control the upper difference bound}.
%\end{remark}

\section{Numerical experiments and discussions}\label{sec:experiment}
Now we present numerical examples to complement the theoretical analysis. We compare OCPs with several commonly employed CPs in terms of optimal convergence factors, and show the parareal
efficiency on three model problems. {Now we recall the concept of parareal efficiency, motivated by the pararell efficiency for an adaptive parareal algorithm in \cite{maday2020adaptive}, and pararell efficiency is a standard concept in high performance computing.} Since OCPs have higher computational complexity than standard ones, e.g., backward Euler scheme, a nuanced evaluation via parareal efficiency is useful. The speed-up, comparing a sequential FP to the parareal algorithm, is defined by
\begin{equation*}
	\textbf{speed-up}_{\text{para/seq}}( \eta ,[ 0,T] ) = {\textbf{cost}_{\text{seq}}( \eta ,[ 0,T] )}/{\textbf{cost}_{\text{para}}( \eta ,[ 0,T] )}\red{,}
\end{equation*}
{where the cost refers to the computing time for the algorithm} 
to achieve a target accuracy $\eta$ for $u^{N_cJ}$, i.e., $\Vert u^{N_cJ}  -\overline{u}^{N_cJ} \Vert  \leq \eta$,
with $\overline{u}^{N_cJ}$ denoting the value of the exact propagator at $t=T$. In Tables \ref{tab:linear: para eff} and \ref{tab:semi linear: para eff} below, we conduct tests both with / without the cost of $\G$ (i.e., the computing time of sequential CPs $\G$).	Then the parareal efficiency is defined as the ratio of the speed-up to the number $N_c$ of processors, i.e.,
\begin{equation*}
	\textbf{eff}_{\text{para/seq}}( \eta ,[ 0,T] ) = \textbf{speed-up}_{\text{para/seq}}( \eta ,[ 0,T] )/{N_c}.
\end{equation*}
 
{For each parareal iteration, we record the time cost of sequential CPs and FPs, excluding the setup cost. The total time cost for a single parareal iteration is the sum of the time cost of sequential CPs and the maximum time cost of FPs across all coarse intervals. Finally, the overall computational cost is obtained by summing the cost across all iterations.} We have ignored the communication cost {and there is no actual parallel implementation in the numerical experiments.}
The code for reproducing the numerical experiments is available at
\url{https://github.com/Qingle-hello/OCPs.git}.

\subsection{Comparative study on convergence factor} We compare OCP with backward Euler scheme (BE), and the two-stage, second-order singly diagonally implicit Runge Kutta scheme (SDIRK-22) in terms of the mesh-independent convergence factor.
The Butcher tableaux and stability function of SDIRK-22 are given respectively by
\noindent
\begin{center}
\begin{minipage}{.49\textwidth}
\centering
\begin{tabular}{c | c c }
    $\gamma$ & $\gamma$ & 0 \\
    1 & $(1-\gamma)$ & $\gamma$ \\
    \hline
    & $(1-\gamma)$ & $\gamma$\\
\end{tabular}
\end{minipage}%
\hspace{-25mm}
\begin{minipage}{.49\textwidth}
\centering
$\text{and}$\qquad $ R(s) = \dfrac{(2\gamma - 1) s + 1}{(\gamma s + 1)^{2}},$
\end{minipage}
\end{center}
with $\gamma = \frac{2-\sqrt{2}}{2}$.
We take high-order Lobatto IIIC and Radau IIA as the FPs, and learn the OCP using Algorithm \ref{algo:coarse} with a ratio $J_0 = 16$.

Let $\phi (s) = \kappa (r,R,J,s)$, which varies with the choices of $r$, $R$ and $J$, and take $\phi^\ast=\max_{s\in [0,\infty)} \phi (s)$ as the {optimized} convergence factor (with $s^\ast$ being the maximizer). {The wide range of $J$ allows examining the sensitivity of the optimized convergence factor with respect to $J$, providing further empirical evidences on the restriction $J\in [16,128]$.}
In Table \ref{tab:compare}, we present $\phi^\ast$ for different combinations of CPs and FPs: OCPs significantly outperform both BE scheme and SDIRK-22, as indicated by the much smaller convergence factor $\phi^\ast$, whereas BE and SDIRK-22 are largely comparable. The improvement is also shown in Fig. \ref{fig:app_2_a} for $J\in\{2,4\}$. The figure exhibits pronounced differences in the local maximizer $s^\ast$ for BE and SDIRK-22, contrasting sharply with that for OCPs. See also Figs. \ref{fig:app_2_b} and \ref{fig:app_2_c} in the appendix for additional plots of $\phi(s)$ with larger $J$ values. Overall, the observation is very similar. Moreover, the extremum value is about 0.014 for all the plots, indicating the robustness of the OCP with respect to the choice of $J$. Furthermore, with the increase of $J$, the variations in $\phi^\ast$ and the maximizer $s^\ast$ diminish, which agrees with the observation from Fig. \ref{fig:lower bound}. In addition, the sharp kinks on the plots (with value close to zero) indicate that any fixed ratio $J$, for certain frequencies $s$ (i.e., for ODEs), the OCP may exhibit a super-fast convergence behavior. However, such kinks are very narrow, and not robust with respect to the problem parameter settings.

\begin{table}[hbt!]
	\centering
 \begin{threeparttable}
  \caption{The maximum $\phi^\ast$ (optimized convergence factor) and maximizer $s^*$ of the function $\phi(s)=\kappa(r, R, J, s)$ for three CPs and four FPs.}
	\label{tab:compare}
	\begin{tabular}{c | c c c c c c c }
		\toprule
	{CP}	&{FP} & Stage & Order & Pad\'e approx. &   & \\
		\midrule
	 	&{Lobatto IIIC }& 2 & 2 & $(0,2)$ &   &  \\
		
		& $J$ &  2 &  4 &  16 &  64 &  128 \\

    	BE&$\phi^\ast$ & 0.264 & 0.287 &  0.298 &  0.298 &  0.298 \\
		&$s^\ast$ & 1.65    &  1.74    &  1.79  & 1.79   &  1.79   \\
		\hline

  SDIRK-22&$\phi^\ast$ & 0.269 & 0.263 &  0.262 &  0.262 &  0.262 \\
		&$s^\ast $ & 7.65  &  8.01  &  8.16  & 8.17   & 8.17   \\
		\hline

		OCP&$\phi^\ast$ & \bf{0.020} & \bf{0.014} &  \bf{0.013} &  \bf{0.014} &  \bf{0.014} \\
		&$s^\ast $ & 0.73    &  0.44    &  10.02  & 2.37   &  2.37   \\
		\midrule
& Lobatto IIIC & 3 & 4 & $(1,3)$ &   &  \\
		& $J$ &  2 &  4 &  16 &  64 &  128 \\
		BE &$\phi^\ast$ &  0.299 &  0.299 &  0.298 &  0.298 &  0.298 \\
		&$s^\ast$ &  1.80    & 1.79   & 1.79 & 1.79  & 1.79    \\
		\hline

  SDIRK-22&$\phi^\ast$ &  0.261 &  0.261 &  0.262 &  0.262 &  0.262 \\
		&$s^\ast$ &  8.26   & 8.18  & 8.17  & 8.17  &  8.17 \\
		\hline

  OCP&$\phi^\ast$ &  \bf{0.014} &  \bf{0.014} &  \bf{0.014} &  \bf{0.014} &  \bf{0.014} \\
		&$s^\ast $ &  2.45    & 0.39     & 0.39     & 0.39     &  0.39    \\
		\midrule
		& Lobatto IIIC& 4 & 6 & $(2,4)$ &   & \\
		& $J $&  2 &  4 &  16 &  64 &  128 \\
		BE&$\phi^\ast$ &  0.298 &  0.298 &  0.298 &  0.298 &  0.298 \\
		&$s^\ast$ & 1.79    & 1.79    &1.79   & 1.79   & 1.79 \\
		\hline

  		SDIRK-22&$ \phi^\ast$ &  0.262 &  0.262 &  0.262 & 0.262 &  0.262 \\
		&$s^\ast$ & 8.17    & 8.17   & 8.17   & 8.17   & 8.17   \\
		\hline

  		OCP&$\phi^\ast$ &  \bf{0.014} &  \bf{0.014} &  \bf{0.014} &  \bf{0.014} &  \bf{0.014} \\
		&$s^\ast$ & 10.01    & 10.02    & 10.02    & 10.02    & 10.02   \\
\midrule
		&{Radau IIA}& 3 & 5 & $(2,3)$ &  &  \\
		& $J $&  2 &  4 &  16 &  64 &  128 \\
		BE&$\phi^\ast$ &  0.298 &  0.298 &  0.298 &  0.298 &  0.298 \\
		&$s^\ast$ & 1.79    & 1.79    &1.79   & 1.79   & 1.79 \\
		\hline

  		SDIRK-22&$\phi^\ast$ &  0.262 &  0.262 &  0.262 & 0.262 &  0.262 \\
		&$s^\ast$ & 8.18    & 8.17   & 8.17   & 8.17   & 8.17   \\
		\hline

  		OCP&$\phi^\ast$ & \bf{0.014} &  \bf{0.014} & \bf{0.014} & \bf{0.014} &   \bf{0.014} \\
		&$s^\ast$ & 10.26   & 10.00    &  10.00    &  10.00    &  10.00   \\
\bottomrule
	\end{tabular}
\end{threeparttable}
\end{table}

\begin{figure} [hbt!]
\centering
\includegraphics[width=1\textwidth,trim={3cm 2cm 3cm 2cm},clip]{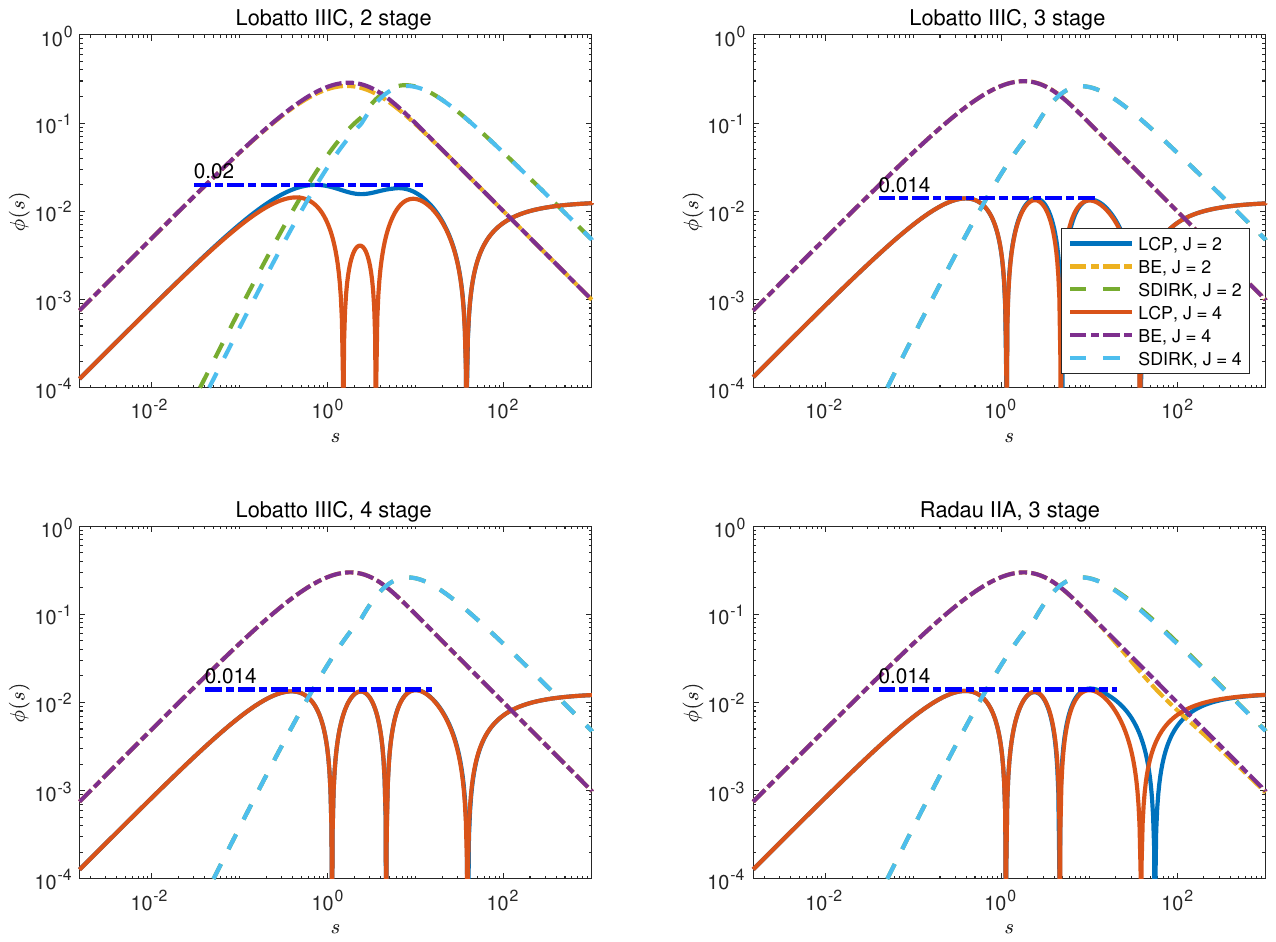}
\caption{The function $\phi (s)$ for three CPs and four FPs when $J\in\{2,4\}$.}\label{fig:app_2_a}
\end{figure}

So far all the employed FPs are L-stable. Now we optimize the CP for an A-stable FP, i.e., the $\theta$ scheme with $\theta \in (1/2,1])$. Its stability function $r(s)$ is given by
\begin{equation*}
    r(s) = \frac{1+(\theta-1)s}{1+\theta s}.
\end{equation*}
Note that $r(s)$ is monotonically decreasing on $\left( 0,\infty\right)$ and $r(\infty) = \frac{\t-1}{\t}>-1$. Using Algorithm \ref{algo:coarse}, we find that the optimized convergence factor is around $0.10$ for $\theta = 0.52,m=n=3, q=1$ and $J=16$, and the corresponding OCP is given in Appendix \ref{app:coarse}. In comparison, the convergence factor is around $0.298$ when BE is employed as CP. So the OCP still significantly outperforms that BE CP. Since $\t$ scheme is seldom used in the parareal algorithm, the numerical experiments below focus on the Gauss-Radau and Gauss-Lobatto type FPs.

\subsection{Illustration on model three problems}\label{ssec: model problems}
Now we illustrate the efficiency of OCPs on three model problems to show its potential. The first one is one-dimensional linear diffusion models.
\begin{example}\label{exam:diffusion}
 Consider the following initial-boundary value problem
\begin{equation}\label{eqn:diffusion}
	\left\{	\begin{aligned}
		\partial_t u(x,t) - \partial_{xx} u(x,t) &= f(x,t), && x \in \Omega, \, 0 < t < T, \\
		u(x,t) &= 0, && x \in \partial\Omega, \, 0 < t < T, \\
		u(x,0) &= u^0(x), && x \in \Omega,
	\end{aligned}\right.
\end{equation}
with $\Omega=(0,\pi)$, and the following three sets of problem data: {\rm(a)} $T=1$, $u^0(x) = x^{10}(x-\pi)^{10}/(\pi/2)^{10}$
and $f\equiv0$; {\rm(b)} $T=1$, $u^0(x) = \chi_{(0,\pi/2)}(x)$ and $f=\cos(t) \sin (x)$, where $\chi_{(0,\pi/2)}(x)$ denotes the characteristic function of the set $(0,\pi/2)$; {\rm(c)} $T=100$, $u^0(x) = 2\chi_{(0,\pi/2)}(x) - 1$ and $f=50\sin(2\pi (x+t))$.
\end{example}

In the experiment, we divide the domain $\Omega$ into $M$ equal subintervals, each of length $h=\pi/M$,
apply the Galerkin FEM with piesewise linear FEM with a mesh size $h=\pi/1000$, and initialize $U^{n}_{0}$ with the CP. Due to the discontinuity of the initial condition in case (b), we have chosen a fine grid for both time and space to accurately resolve the solution. Throughout, we study the error between the iterative solution $U_k^n$ by the parareal algorithm and the fine time stepping solution $u^{nJ}$, i.e.,
$\text{error}=\max_{1\leq n\leq N_{c}}\|U^{n}_{k}-u^{nJ}\|_{L^{2}(\Omega)} $.
	
The problem data in case (a) is smooth and compatible with the zero
Dirichlet boundary condition. In Fig. \ref{fig:ex1 smooth}, we show the convergence rate of
the parareal algorithm with OCPs for two, three, four-stage Lobatto IIIC
methods and three-stage Radau IIA method with $J=20$, $100$ and $\Delta t = 1/500$ (i.e., $\Delta T
=1/25$ and $1/5$), and take the BE scheme as the benchmark CP. The
convergence property of the parareal algorithm, with BE as the CP, has been studied \cite{yang2021robust}. Note that when $J=100$, the parareal algorithm should converge to the fine solution by the 5th iteration. Indeed, the errors drop to 0, confirming the theoretical prediction (we replace zero with $1.0 \cdot 10^{-15}$ in the error plot). OCPs perform much better
than BE and is also robust with respect to $J$. The convergence rate of the parareal algorithm
with the BE scheme is around 0.3 when $J=100$, as predicted by \eqref{eqn:conv-1}.

The initial data $u_0$ in case (b) is nonsmooth. From {Fig. \ref{fig:ex1 robust}}, the convergence rate of the parareal
algorithm agrees well with the theoretical predictions from Theorem \ref{thm: estimate for para}.
Note that for $J=20$, the parareal algorithm with OCPs is two steps faster than that with the BE scheme
(for the error tolerance $1.0\cdot10^{-8}$). The efficiency is accentuated for $J=100$: the learned algorithm is nine steps faster than the standard one. 	

In Table \ref{tab:linear: para eff}, we show the efficiency of the parareal algorithm for case (c), with $J=100$ and a target accuracy $\eta=1.0\cdot10^{-7}$. To achieve the target accuracy, we vary the fine step size with the method, and also exclude the two-stage Lobatto IIIC method due to its low accuracy. {We have listed the corresponding number of parareal iterations needed to achieve the given accuracy next to each FP in Table \ref{tab:linear: para eff}, indicated by the number in the bracket}. {The time $t_n$ increases with the order of the method increases, but the speed-up decreases due to a shorter time needed to reach the desired accuracy. Meanwhile, the efficiency improves and the cost of $\G$ reduces, since maintaining the ratio $J$ requires fewer processors.} The comparison between OCPs and BE indicates that despite the increase in its complexity, OCPs better suit the parareal algorithm for linear evolution problems.

\begin{figure} [hbt!]
\centering
\includegraphics[width=0.95\textwidth,trim={3cm 2cm 3cm 2cm},clip]{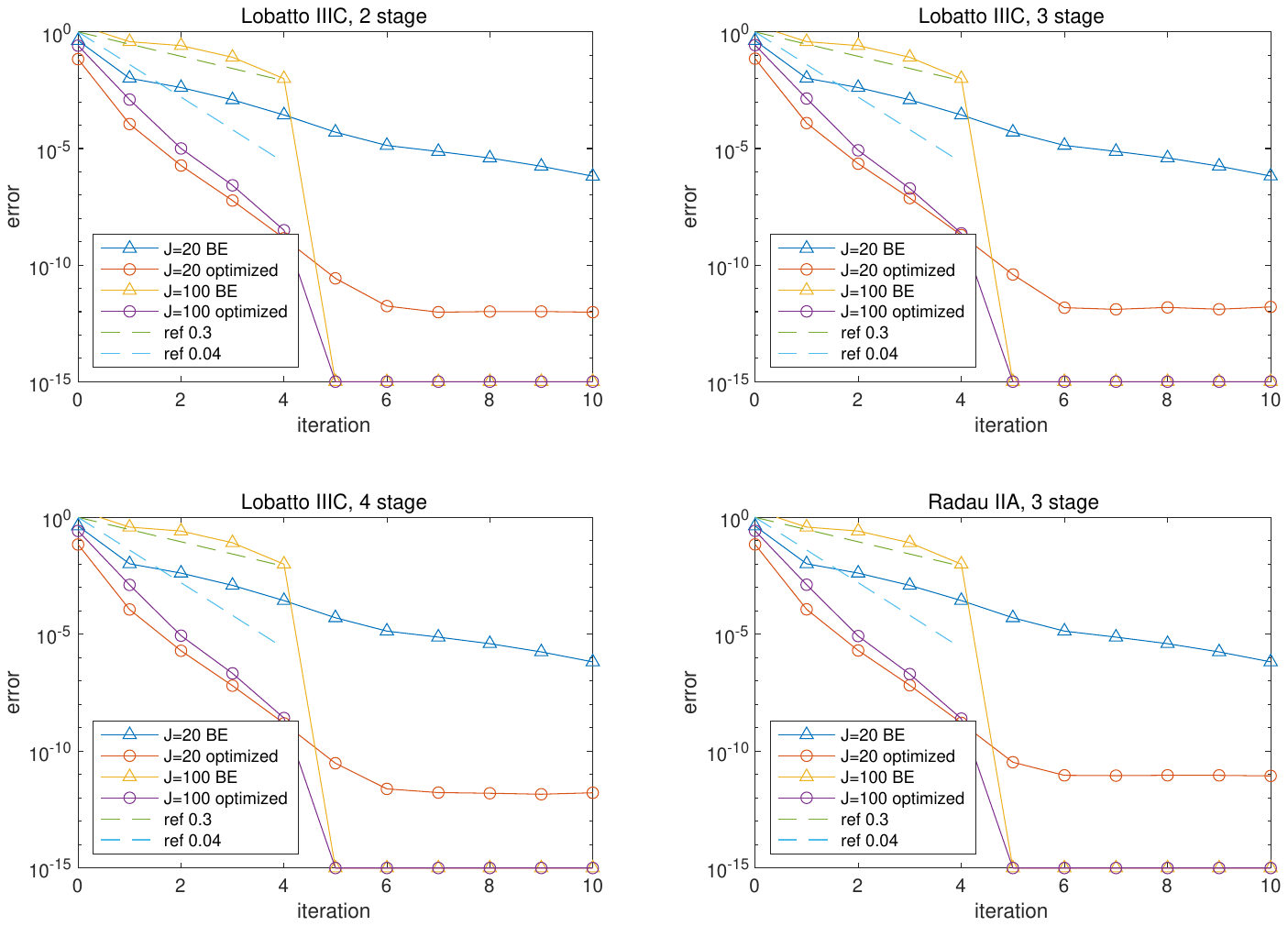}
\caption{The convergence rate for Example \ref{exam:diffusion}(a) (smooth data).
\label{fig:ex1 smooth}}
\end{figure}

\begin{figure} [hbt!]
\centering
 \includegraphics[width=0.95\textwidth,trim={3cm 2cm 3cm 2cm},clip]{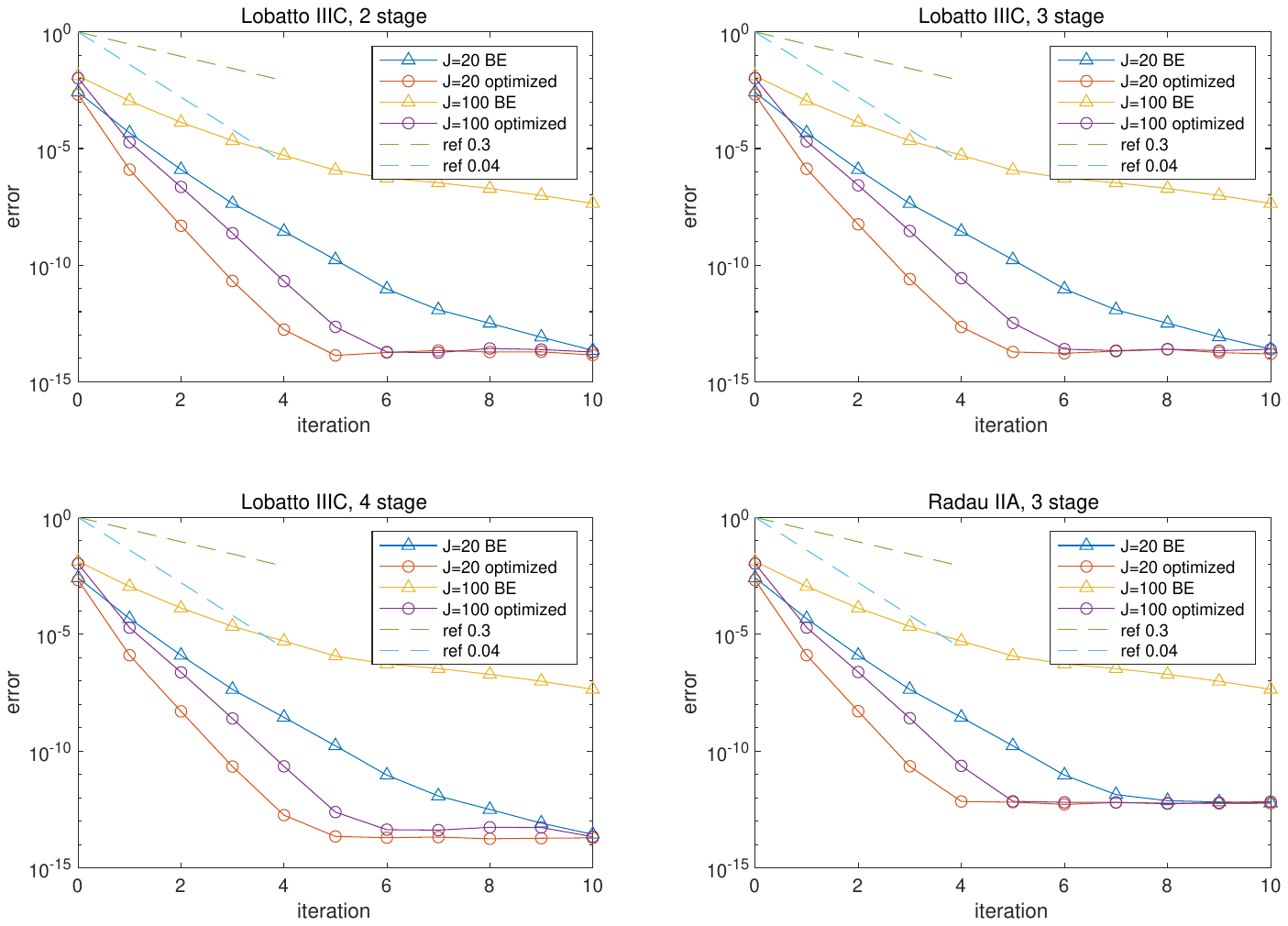}
\caption{The convergence rate for Example \ref{exam:diffusion}(b) (nonsmooth data).\label{fig:ex1 robust}}
\end{figure}

\begin{table}
\centering
\begin{threeparttable}
\caption{The parareal efficiency for Example \ref{exam:diffusion}(c) (nonsmooth data), with $J=100$ and {$\eta = 1.0\cdot10^{-7}$}.}
	\begin{tabular}{c|cccccc}
			\toprule
		& \multicolumn{2}{c}{three-stage Lobatto} & \multicolumn{2}{c}{three-stage Radau } & \multicolumn{2}{c}{four-stage Lobatto} \\ \midrule
		 $\Delta t$ & \multicolumn{2}{c}{1/150} & \multicolumn{2}{c}{1/45} & \multicolumn{2}{c}{1/30} \\ \hline
			\textbf{Speed-up}  & BE \red{(10)}& OCP \red{(4)}& BE \red{(11)}& OCP \red{(4)}& BE \red{(10)}& OCP \red{(4)}\\ \hline
			w cost $\G$ &11.48 & 25.52 & 4.81 & 11.03 & 3.93 & 10.35 \\
w/o cost $\G$  & 11.51 & 25.67 & 4.82 & 11.05 & 3.94 & 10.36 \\
\midrule
			\textbf{Efficiency}  & BE & OCP & BE & OCP & BE & OCP \\
\midrule
w  cost $\G$  & 7.65\% & 17.02\% & 10.68\% & 24.51\% & 13.09\% & 34.51\% \\
			w/o  cost $\G$  &7.68\% & 17.11\% & 10.71\% & 24.55\% & 13.10\% & 34.55\% \\
   \bottomrule
		\end{tabular}
			\label{tab:linear: para eff}
    \end{threeparttable}
\end{table}

Next we illustrate the method with a semilinear parabolic problem, the Allen--Cahn equation. It was proposed to describe the motion
of anti-phase boundaries in crystalline solids \cite{AllenCahn:1972}, where $u$ denotes the concentration of one of the two metallic components
of the alloy and $\epsilon>0$ controls the width of interface. %Currently it is popular for describing moving interface problems in materials science and fluid dynamics.
\begin{example}\label{exam:AC}
Consider the initial-boundary value problem of the Allen-Cahn equation:
\begin{equation}\label{eqn:AC}
	\left\{
	\begin{aligned}
		\partial_t u - \partial_{xx} u &= \epsilon^{-2}(u - u^3), && x\in \Omega, \,t\in(0,T], \\
		\partial_x u(x,t), &= 0 && x\in \partial \Omega,\, t\in(0,T], \\
		u(x,0) &= u_0(x), && x\in \Omega,
	\end{aligned}
	\right.
\end{equation}
with $\Omega=(0,\pi)$, $T=1$ and $u_0=1-2\chi_{(\pi/2,\pi)}$.
\end{example}

In the experiment, we divide the domain $\Omega$ into 1000 equal subintervals of length $h=\pi/1000$ and apply the finite difference method  in space.
We take the CP $\mathcal{G}$ to be the first-order semi-implicit backward scheme: given $u^n$, find $u^{n+1}$ such that
\begin{equation*}
	u^{n+1}=R( \Delta T \A)  u^{n}+\Delta TP( \Delta T \A)  f( {}u^{n})  ,\quad \text{with } f(u) = \epsilon^{-2}(u - u^3).
\end{equation*}
The FP $\mathcal{F}$ is a fully implicit, high-order, single step integrator, e.g., Lobatto IIIC or fully implicit Radau schemes: given $u^n$, we find $u^{n+1}$ such that \begin{equation}\label{eqn:semi-implicit}
	\begin{cases}
       u^{ni}=u^{n-1}+\Delta t\sum^{m}_{j=1} a_{ij}( \A u^{nj}+f( u^{nj})  )  ,&\\
       u^{n}=u^{n-1}+\Delta t\sum^{m}_{i=1} b_{i}( \A u^{ni}+f( u^{ni})  )  .&
    \end{cases}
\end{equation}
Here, $u^{ni}$ denotes the interstage of the implicit single step integrator. If the time step size $\Delta t$ is small, the nonlinear system \eqref{eqn:semi-implicit} is uniquely solvable, e.g., via Newton's algorithm. Since the FP $\mathcal{F}$ is fully nonlinear, it is time consuming to numerically realize the scheme, whereas the CP $\mathcal{G}$ is a linear scheme, so the parareal algorithm may greatly improve the computational efficiency.
	
\begin{figure} [tbhp]
\centering\includegraphics[width=0.95\textwidth,trim={3cm 2cm 3cm 2cm},clip]{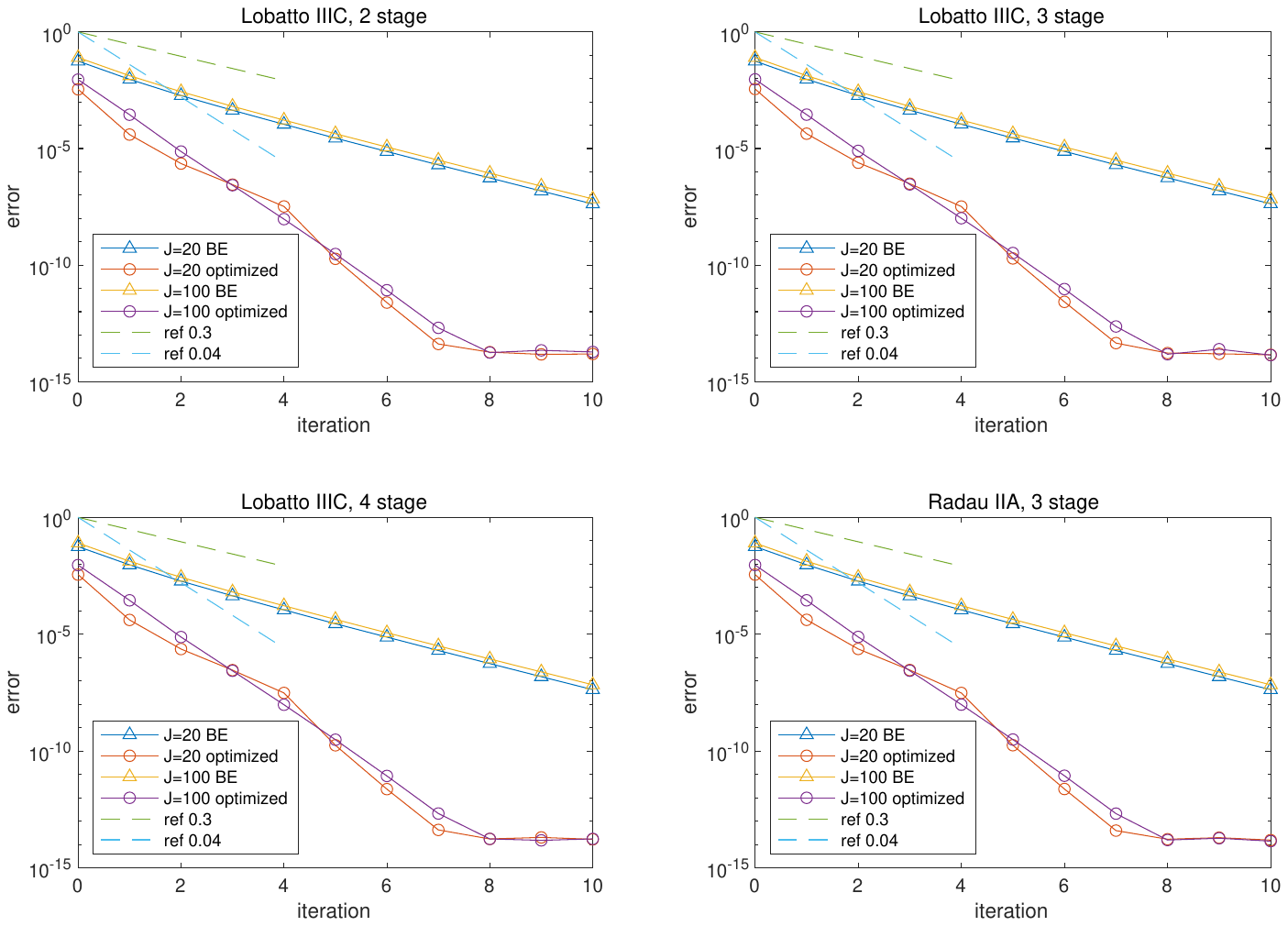}
\caption{The convergence rate for Example \ref{exam:AC} (Allen--Cahn equation) with $\epsilon^{2}=1$.}
	\label{fig:ex2_1}
\end{figure}
	
In the experiment, we take $\epsilon^{2}\in\{ 1, 0.02 \}$. In Fig. \ref{fig:ex2_1}, for $\epsilon^{2} = 1$, we illustrate OCPs with $J=20$, $100$ and $\Delta t = 1/2000$ (i.e., $\Delta T=1/100$ and $1/20$). The results show that the convergence rate of the parareal algorithm, using OCPs, is about $0.04$, compared to $0.3$ with the BE scheme. Moreover, an increase in $J$ corresponds to a larger $\Delta T$, limiting the CP's capability to
accurately capture the solution structure. Thus, the convergence rate deteriorates as $J$ increases. Note that for the BE CP, there exists a critical threshold $J_{\star} > 0$, such that the parareal solver converges with a convergence rate near $0.3$ whenever the ratio $J$ satisfies $J \geq J_{\star}$, including nonsmooth problem data \cite{yang2021robust}. The threshold $J_*$ is estimated to be $J_{\star} = 2$. Thus the error behaviors of the BE CP for $J=20$ and $J=100$ are nearly identical $J>J_{\star} = 2$ and the choice $J$ has little influence on the convergence, confirming the theoretical prediction.  In Fig. \ref{fig:ex2_2}, for $\epsilon^{2}=0.02$, we illustrate OCPs with $J=10,$ $20$ and $\Delta t = 1/10000$ (i.e., $\Delta T=1/1000$ and $1/500$).
For $J=10$, $20$, the convergence rate of OCPs is about $0.04$. However, for larger $J$, the difference between the BE scheme and OCPs is small, calling for modifications of OCPs for nonlinear problems, especially for small $\epsilon$. In Table \ref{tab:semi linear: para eff}, we show the parareal efficiency for $\epsilon^{2}=1$, $T=10$, $J=20$ and
a target accuracy {$\eta={1.0\cdot10^{-6}}$}. 
{In the table, to compensate for nonlinearity, we have experimented with smaller coarse step sizes $\delta T=\Delta T/5$ for both OCPs and BE. Since the cost associated with solving nonlinear systems in fine steps is significant, the cost of explicit CPs is negligible. This is evidenced by the small difference in speed-up and efficiency between scenarios with / without cost $\G$, cf. Table \ref{tab:semi linear: para eff}.} The trend of speed-up and efficiency is consistent with Table \ref{tab:linear: para eff}. However, the strong nonlinearity renders the efficiency improvement less pronounced.

\begin{figure} [tbhp]
\centering
\includegraphics[width=0.95\textwidth,trim={3cm 2cm 3cm 2cm},clip]{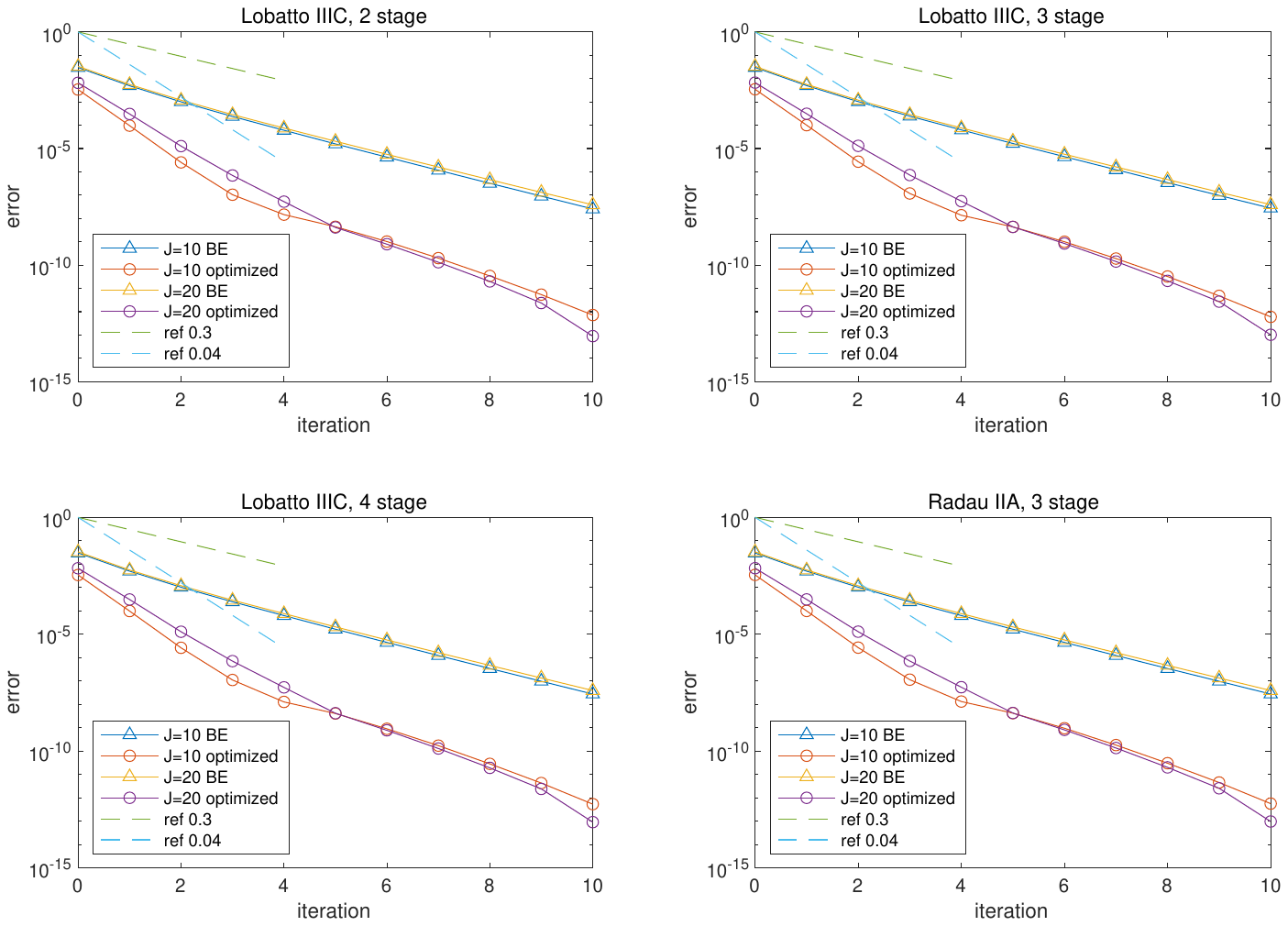}
\caption{The convergence rate for Example \ref{exam:AC} (Allen--Cahn equation) with $\epsilon^{2}=0.02$.\label{fig:ex2_2}}
\end{figure}
	
\begin{table}
\centering
\begin{threeparttable}
\caption{The parareal efficiency for Example \ref{exam:AC} with $\epsilon^{2}=1$, $T=10$, $J=20$ and {$\eta ={1.0\cdot10^{-6}}$}.}
		\begin{tabular}{c|cccccc}
		\toprule
		& \multicolumn{2}{c}{three-stage Lobatto} & \multicolumn{2}{c}{three-stage Radau } & \multicolumn{2}{c}{four-stage Lobatto} \\ \hline
		 $\Delta t$ & \multicolumn{2}{c}{1/300} & \multicolumn{2}{c}{1/100} & \multicolumn{2}{c}{1/80} \\ \midrule
		\textbf{Speed-up}  & BE \red{(6)} & OCP \red{(3)} & BE \red{(11)} & OCP \red{(5)} & BE \red{(11)} & OCP \red{(3)} \\ \midrule
	w cost $\G$ &10.93 & 18.78 & 3.65 & 7.90 & 3.43 & 11.63 \\
	w/o cost $\G$  & 11.62 & 19.70 & 3.74 & 8.40 & 3.48 & 11.90 \\ \midrule
		\textbf{Efficiency}  & BE & OCP & BE & OCP & BE & OCP \\ \midrule
	w cost $\G$  & 7.29\% & 12.52\% &7.27\% & 15.80\% & 8.59\% & 29.08\% \\
w/o cost $\G$  &7.75\% & 13.13\% & 7.48\% & 16.80\% & 8.69\% & 29.75\% \\
 \bottomrule
	\end{tabular}
	\label{tab:semi linear: para eff}
\end{threeparttable}
\end{table}
	
The third and last example is about viscous Burgers' equation, which is a model problem in fluid dynamics,
describing e.g., the decay of turbulence within a box. %$u(x,t)$ measures the fluid viscosity.	
\begin{example}\label{exam:Burgers}
Fix $T=1$, and $u_{0}(x)=\chi_{(0,1/2)}(x).  $  Consider  the viscous Burgers' equation
\begin{equation*}
	\left\{	\begin{aligned}
		\partial_{t} u + u \partial_{x} u &= \nu \partial_{xx} u, & &x \in (0,1), \, t \in (0,T], \\
		u(0,t) &= u(1,t) = 0, & &t \in (0,T], \\
		u(x,0) &= u_{0}(x), & &x \in (0,1).
	\end{aligned}\right.
\end{equation*}
\end{example}

We take $\nu \in \{1, 0.02 \}$ in the experiment. In Fig. \ref{fig:ex3_a}, for $\nu =1$, we examine OCPs with $J=20$, $100$ and $\Delta t = 1/2000$ (i.e., $\Delta T=1/100$ and $1/20$). The results show that the convergence rate of the parareal algorithm, with OCPs, is approximately $0.04$, which agree well with that in Fig. \ref{fig:ex2_1}.
In {Fig. \ref{fig:ex3_b}}, for the more challenging case $\nu =0.02$, we evaluate OCPs with $J=10, 100$ and $\Delta t = 1/10000$ (i.e., $\Delta T=1/1000$ and $1/100$). For $J=10$, the convergence rate of OCPs is approximately 0.06, whereas that for BE is around 0.21. Remarkably, the ratio $\Delta T / \nu $ is unchanged for $J=100$ in the linear case, cf. {Fig. \ref{fig:ex1 smooth}}, so their behaviors are similar. This phenomenon has been analyzed for linear advection-diffusion cases in \cite{gander2020reynolds}. For $J=100$,  both the parareal algorithm and the standard one exhibit slower convergence. Nonetheless, OCP is still competitive with the BE CP.

\begin{figure} [hbt!]
\centering
\includegraphics[width=0.95\textwidth,trim={3cm 2cm 3cm 2cm},clip]{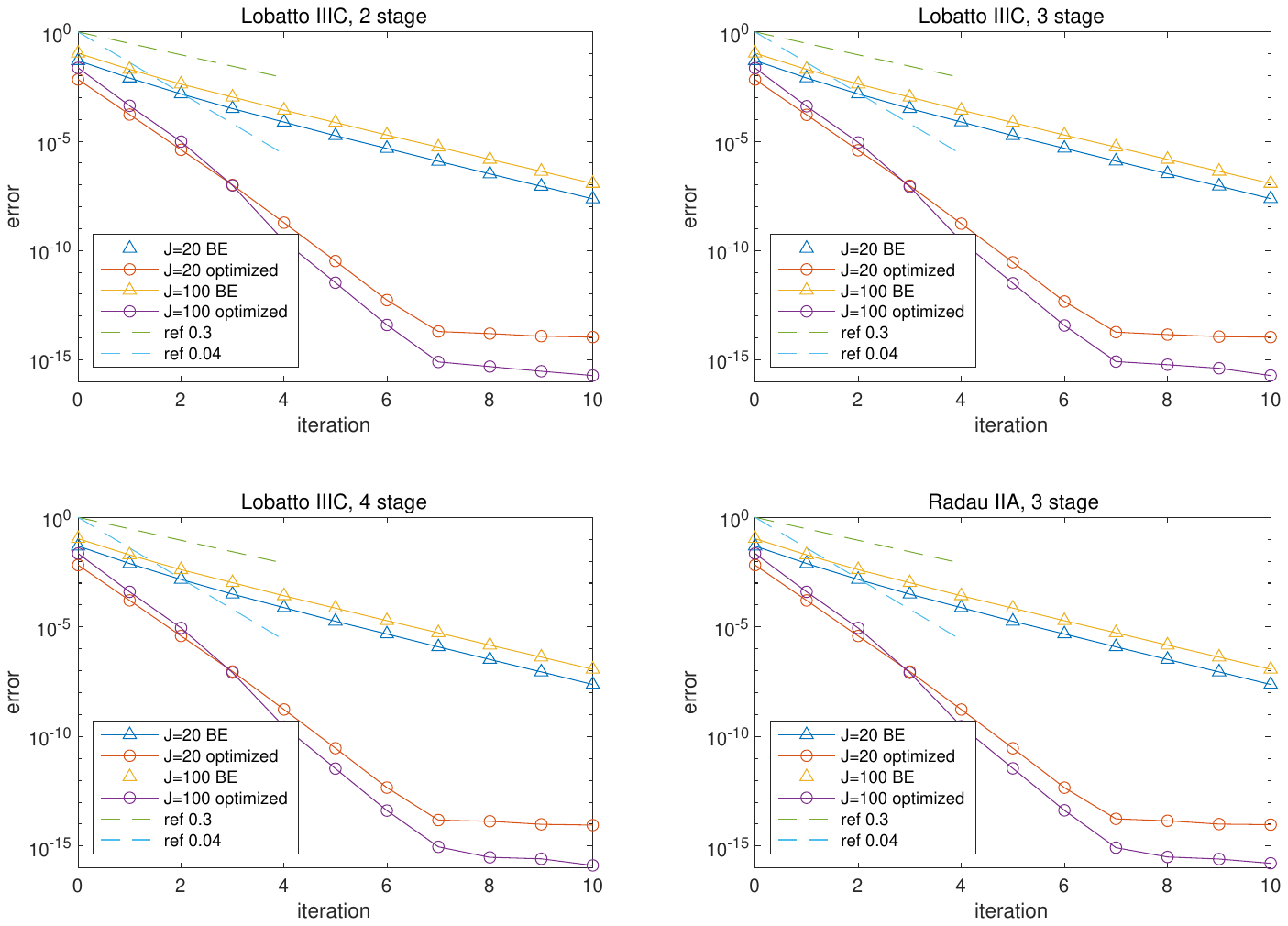}
\caption{The convergence rate for Example \ref{exam:Burgers} (viscous Burgers' equation) with $\nu = 1$.}
\label{fig:ex3_a}
\end{figure}
	
\begin{figure} [hbt!]
\centering
\includegraphics[width=0.95\textwidth,trim={3cm 2cm 3cm 2cm},clip]{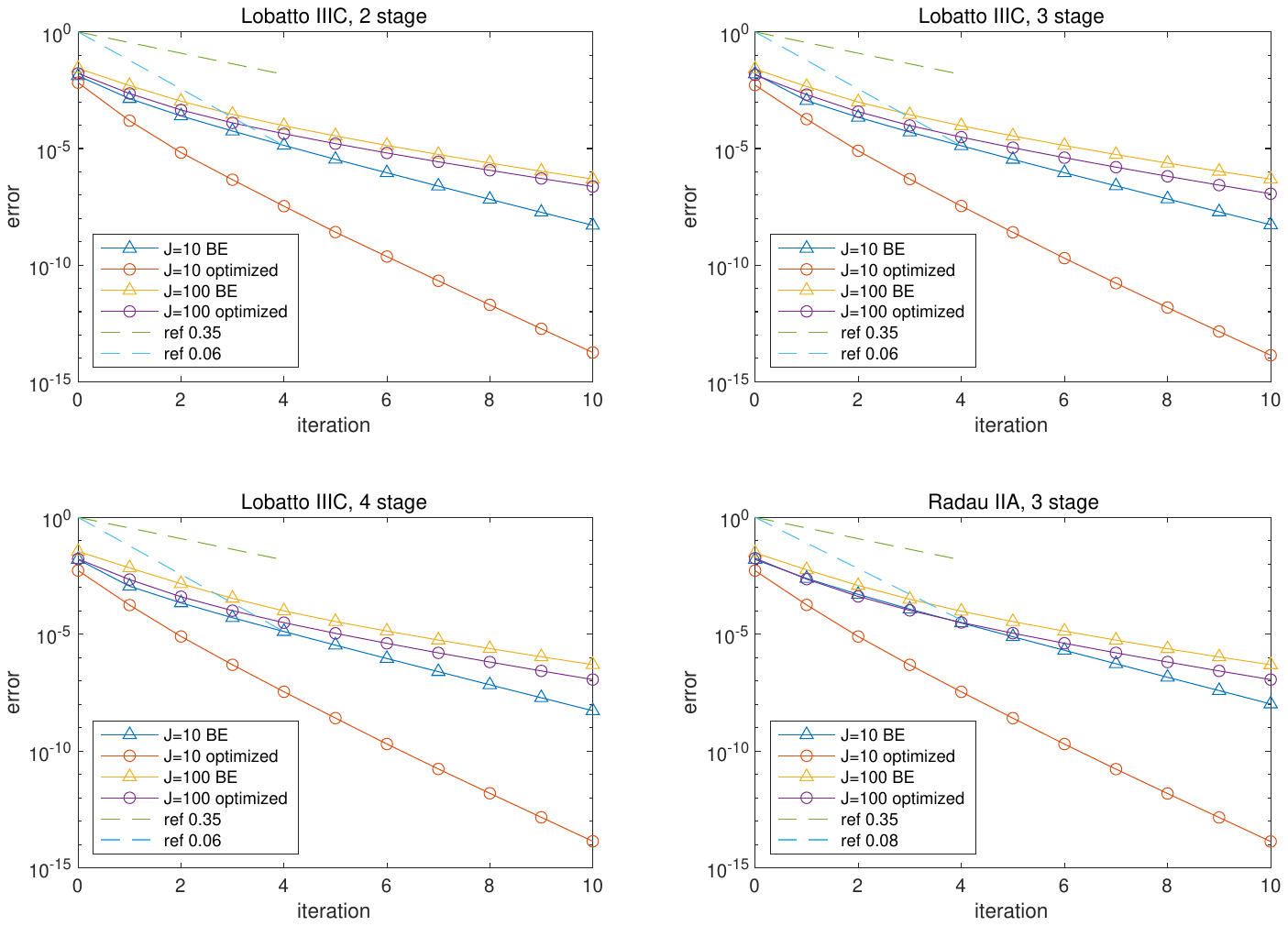}
\caption{The convergence rate for Example \ref{exam:Burgers} (viscous Burgers' equation) with $\nu = 0.02$.}
\label{fig:ex3_b}
\end{figure}

\section{Conclusion}
We have explored the potential of optimizing coarse propagators in parareal algorithms to overcome the barrier in convergence order. The approach is based on rigorous error estimation for linear evolution problems, which ensures that the OCPs enjoy highly desirable properties, e.g., stability and consistent order. We have developed a systematic procedure for optimizing the CP. The numerical experiments fully confirm the theoretical prediction, including mildly nonlinear problems. Most remarkably, even for highly nonlinear problems, the OCPs can still outperform the classical BE scheme.

One important question is to extend the idea to far more challenging situations, e.g., highly nonlinear problems, multiscale problems \cite{LiHu:2021}, {and hyperbolic problems \cite{DeSterck:2021,DeSterck:2024}. The key of the extensions lies in the error estimation, which remains largely open for, e.g., general hyperbolic problems.} Also it is promising to develop the optimizing strategy for other parallel-in-time algorithms, e.g., the two-level MGRIT algorithm with FCF-smoother  \cite{Southworth:2019,FriedhoffSouthworth:2021}. This may necessitate revisiting the formulation of the optimization problem, e.g., inclusion of the constraints \eqref{eq:accurate}.

\section*{Acknowledgements}
The authors are grateful to two anonymous referees for their constructive comments which have led to improved presentation of the paper.
 
\appendix
\section{Additional plots} Figs. \ref{fig:app_2_b} and \ref{fig:app_2_c} show additional plots of the function $\phi(s)$ for different choices of the ratio $J$, complementing Fig. \ref{fig:app_2_a}.

\begin{figure} [hbt!]
\centering
\includegraphics[width=1\textwidth,trim={3cm 2cm 3cm 2cm},clip]{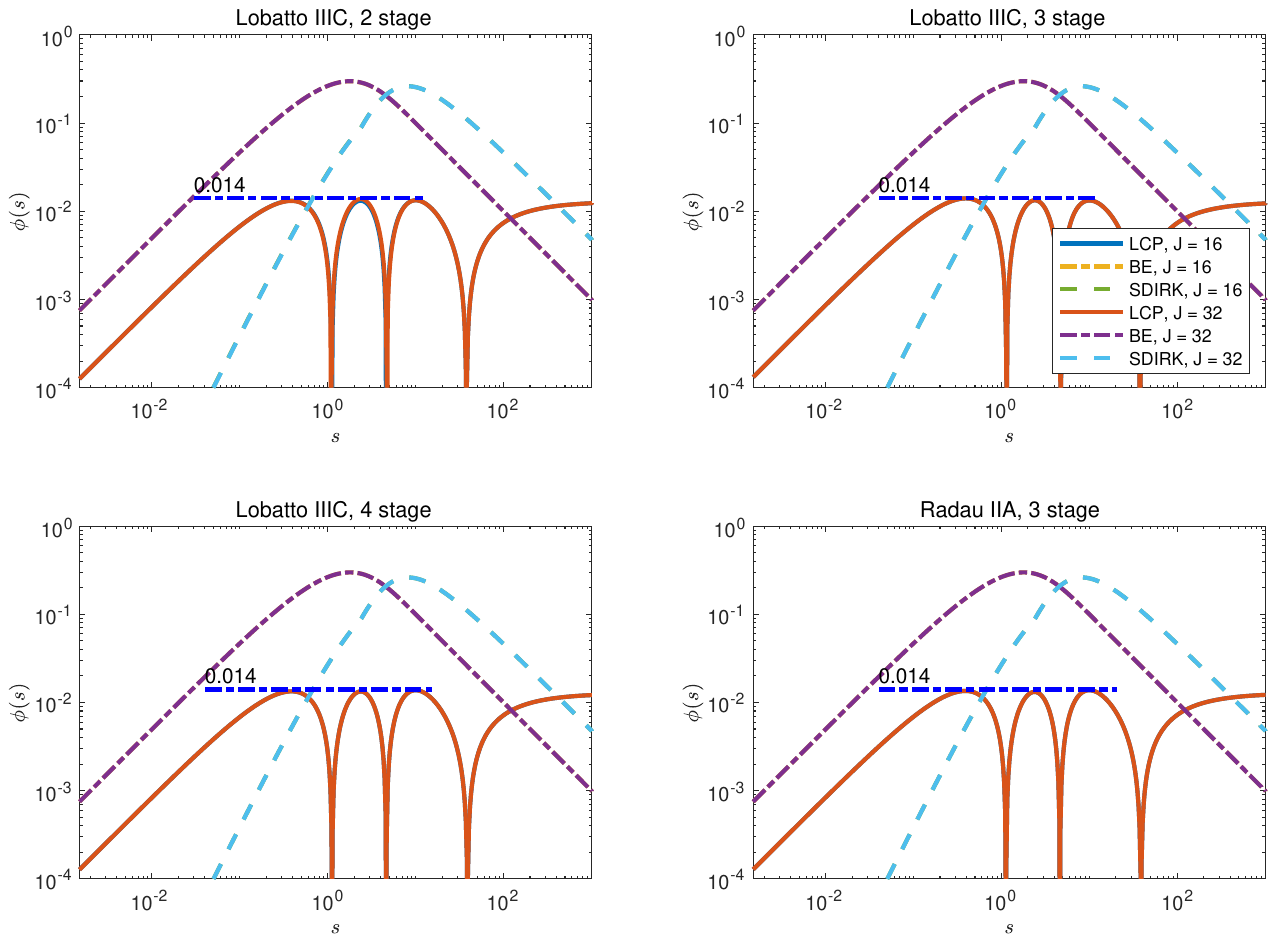}
\caption{ The function $\phi (s)$ for three CPs and four FPs when $J\in\{16,32\}$.}\label{fig:app_2_b}
\end{figure}

\begin{figure} [hbt!]
\centering
\includegraphics[width=1\textwidth,trim={3cm 2cm 3cm 2cm},clip]{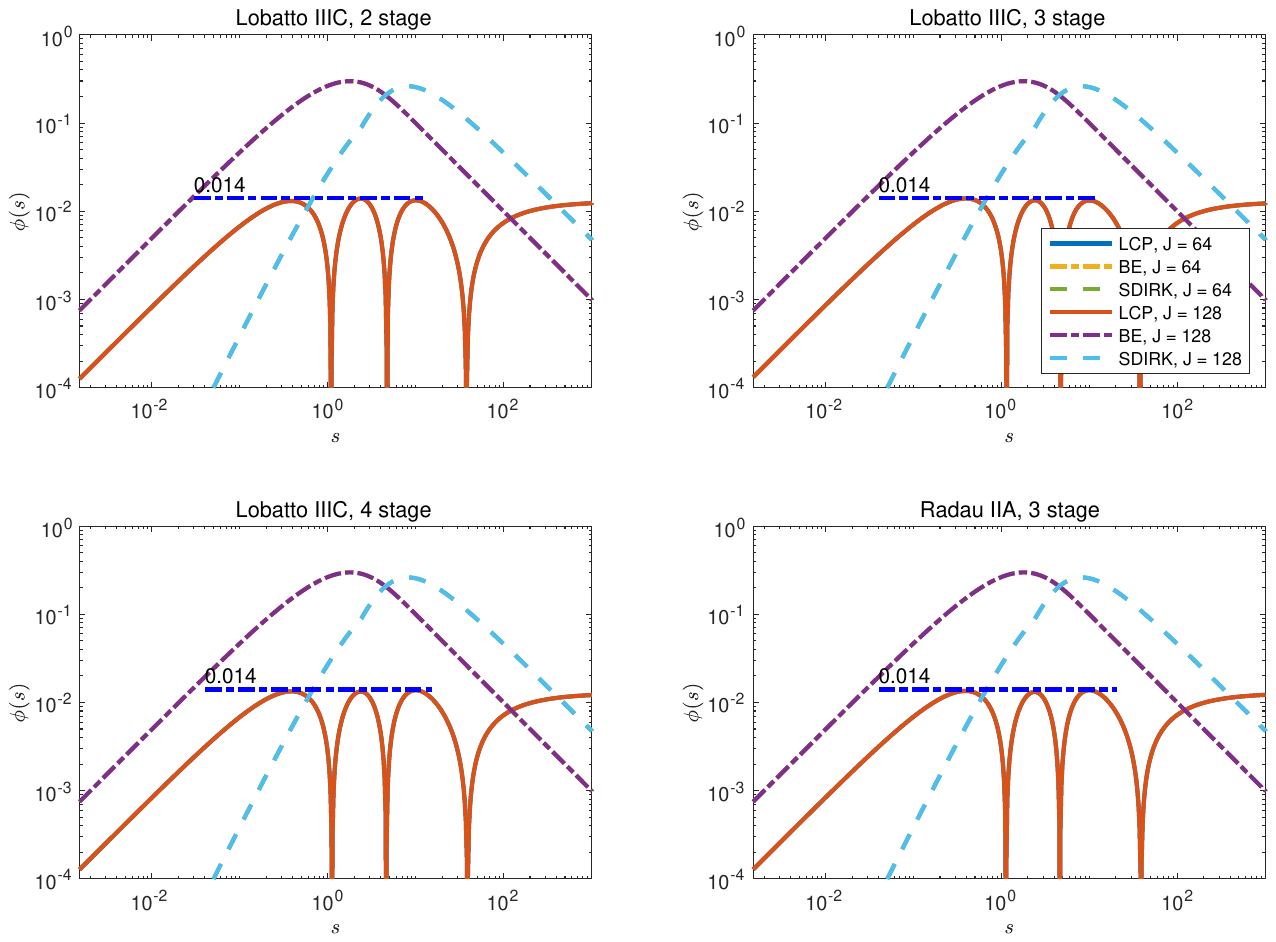}
\caption{ The function $\phi(s)$ for three CPs and four FPs when $J\in\{64,128\}$.}\label{fig:app_2_c}
\end{figure}

\section{Explicit formulas of OCPs}\label{app:coarse}
In this appendix, we present the explicit formulas of OCPs, for $q=1$, $C_{1}^{0}=1$ and $J_0=16$.
\vskip5pt
\noindent (1) For two-stage Lobatto IIIC method,
		\begin{align*}
			R_2(\lambda) &= \frac{{1.0 - 0.20967\lambda + 0.00484\lambda^2}}{{1.0 + 0.79033\lambda + 0.37931\lambda^2}}\quad\mbox{and}\quad 			P_1(\lambda) = \frac{1.0 + 0.37447\lambda}{1.0 + 0.79033\lambda + 0.37931\lambda^2}.
		\end{align*}
		
\noindent (2) For three-stage Lobatto IIIC method,
		\begin{align*}
			R_3(\lambda) &= \frac{{1.0 - 0.21014\lambda + 0.00486\lambda^2}}{{1.0 + 0.78986\lambda + 0.38283\lambda^2}}\quad\mbox{and}\quad
			P_1(\lambda) = \frac{1.0 + 0.37797\lambda}{1.0 + 0.78986\lambda + 0.38283\lambda^2}.
		\end{align*}
		
\noindent (3) For four-stage Lobatto IIIC method,
		\begin{align*}
	R_4(\lambda) &= \frac{{1.0 - 0.21091\lambda + 0.00476\lambda^2}}{{1.0 + 0.78909\lambda + 0.37898\lambda^2}}\quad\mbox{and}\quad
			P_1(\lambda) = \frac{1.0 + 0.37422\lambda}{1.0 + 0.78909\lambda + 0.37898\lambda^2}.
		\end{align*}

\noindent (4) For three-stage Radau IIA method,
		\begin{align*}
			R_r(\lambda) &= \frac{{1.0 - 0.21125\lambda + 0.00479\lambda^2}}{{1.0 + 0.78875\lambda + 0.37922\lambda^2}}\quad\mbox{and}\quad
			P_1(\lambda) = \frac{1.0 + 0.37443\lambda}{1.0 + 0.78875\lambda + 0.37922\lambda^2}.
		\end{align*}
\noindent (5) For $\theta$ scheme ($\t = 0.52$),
		\begin{align*}
		R_{\bt}(\lambda) &= \frac{1.0 + 84.79323\lambda - 15.89302\lambda^{2} + 0.13247\lambda^{3}}{1.0 + 85.79323\lambda + 83.00272\lambda^{2} + 0.66421\lambda^{3}},\\
\text{and}\quad P_{\t}(\lambda) &= \frac{1.0 + 42.18022\lambda + 0.20341\lambda^{2}}{1.0 + 36.15808\lambda + 35.70839\lambda^{2} + 0.25501\lambda^{3}}.
		\end{align*}
In Fig. \ref{fig:app_2}, we show the stability functions of OCPs and the corresponding fine propagators (for the first four cases). It is observed that they are all stable approximations of $e^{-\lambda}$ around $\lambda=0$.
	
\begin{figure} [tbhp!]
\centering
\includegraphics[width=1\textwidth,trim={3cm 2cm 3cm 2cm},clip]{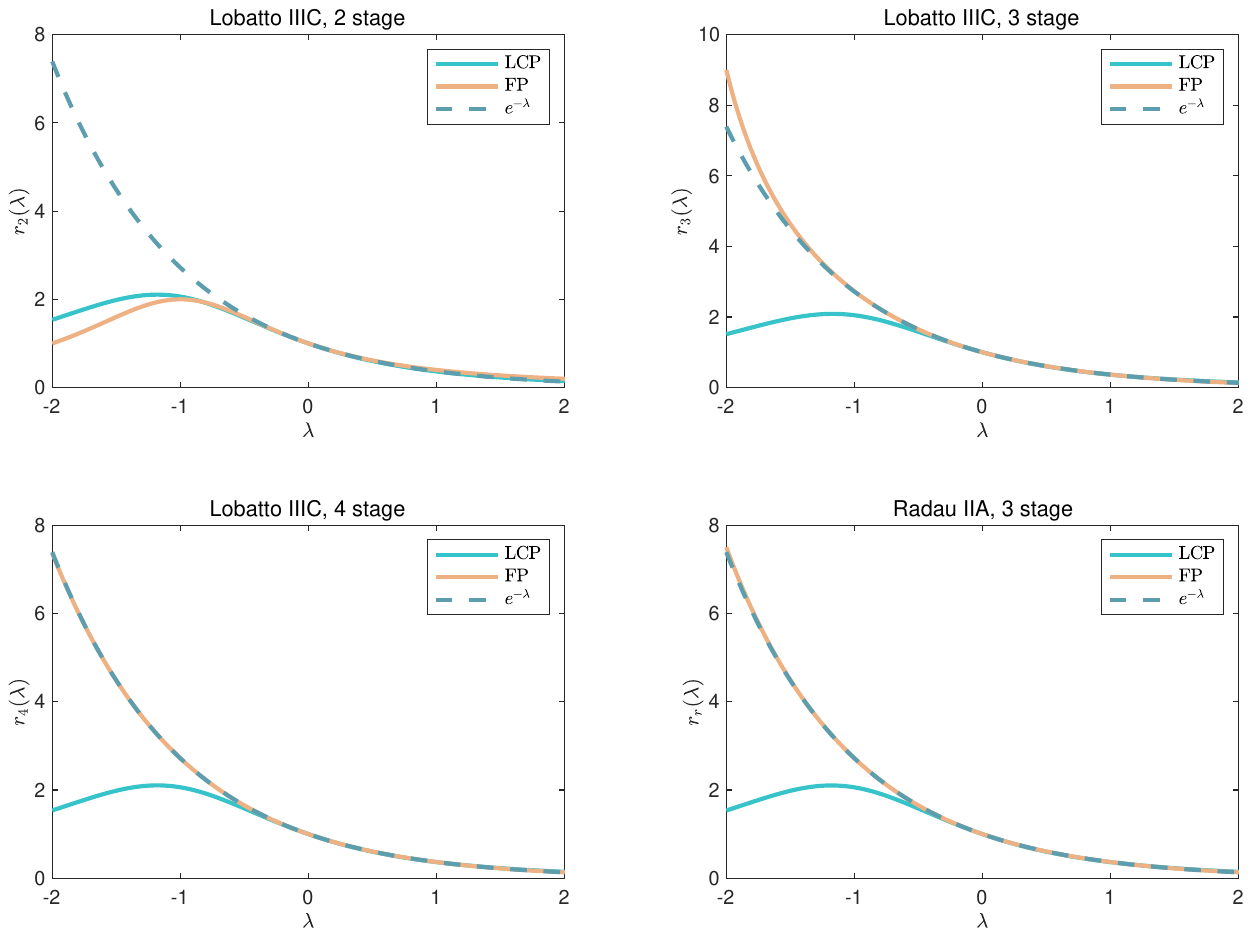}
\caption{The stability functions of the OCPs, for four FPs.}\label{fig:app_2}
\end{figure}

\bibliographystyle{siam}

\end{document}